\documentclass[11pt]{article}
\usepackage{sibarticle}
\usepackage{lmodern}
\usepackage{url}            
\usepackage{nicefrac}       
\usepackage{microtype}      
\usepackage{amssymb,graphicx}
\usepackage{float,amsmath}
\usepackage{amsfonts,epsfig}
\usepackage{subcaption}

\usepackage{algorithm}
\usepackage{makecell}
\usepackage{algpseudocode}
\algblock{ParFor}{EndParFor}
\algnewcommand\algorithmicparfor{\textbf{parfor}}
\algnewcommand\algorithmicpardo{\textbf{do}}
\algnewcommand\algorithmicendparfor{\textbf{end\ parfor}}
\algrenewtext{ParFor}[1]{\algorithmicparfor\ #1\ \algorithmicpardo}
\algrenewtext{EndParFor}{\algorithmicendparfor}

\newcommand{\N}{\mathbb{N}}
\newcommand{\R}{\mathbb{R}}

\DeclareMathOperator*{\argmin}{arg\,min}

\title{Counterfactual Explanations for Linear Optimization} 
\ShortTitle{CELOPT}
\ShortAuthors{Kurtz, Birbil, den Hertog}

\NumberOfAuthors{1}
\FirstAuthor{Jannis Kurtz, Ş. İlker Birbil, Dick den Hertog}
\FirstAuthorAddress{University of Amsterdam, Plantage Muidergracht 12, 1018 TV Amsterdam}


\allowdisplaybreaks

\begin{document}
\setuptodonotes{backgroundcolor=red!10, linecolor=red,
  bordercolor=red, size=\tiny, tickmarkheight=0.1cm}

\newcommand\hlb[3][]{ \todo[inline,caption={emptytext},
  size=\normalsize, backgroundcolor=yellow!70, bordercolor=yellow!70, noshadow, #1]{
    \begin{minipage}{
        \textwidth-4pt}#2
    \end{minipage}}
  \todo{\begin{spacing}{0.5}#3\end{spacing}}}

\newcommand{\hlc}[2]{\hl{#1}\todo{\begin{spacing}{0.5}#2\end{spacing}}}

\newcommand{\sitodo}[2][]{\todo[caption={#2}, #1]{
    \begin{spacing}{0.5}#2\end{spacing}}}

\newcommand{\intodo}[2][]{\todo[inline, noshadow, caption={emptytext}, #1]{
    \begin{spacing}{1.0}\normalsize{#2}\end{spacing}}}

\newcommand{\sib}[1]{\textcolor{red}{#1}}
\newcommand{\tsib}[1]{\begin{tcolorbox}\textcolor{red}{#1}\end{tcolorbox}}

\newcounter{sibcmntcounter}
\setcounter{sibcmntcounter}{1}
\long\def\symbolfootnote[#1]#2{\begingroup
  \def\thefootnote{\fnsymbol{footnote}}\footnote[#1]{#2}\endgroup}
\newcommand{\sibcmnt}[1]{{\textbf{
      \textcolor{red}{(C.\arabic{sibcmntcounter})}}
    \let\thefootnote\relax\footnotetext{\textcolor{red}
        {\small(C.\arabic{sibcmntcounter})~#1}}}
  \addtocounter{sibcmntcounter}{1}}

\newcommand{\red}[1]{\textcolor{red}{#1}}
\newcommand{\blue}[1]{\textcolor{blue}{#1}}
\newcommand{\magenta}[1]{\textcolor{magenta}{#1}}
\newcommand{\bm}[1]{\mbox{\boldmath{$#1$}}}
\newcommand{\rb}[1]{\raisebox{-1.5ex}[0cm][0cm]{#1}}
\newcommand{\HRule}{\noindent\rule{\linewidth}{0.5mm}}
\newcommand{\dsum}{\displaystyle\sum}
\newcommand{\veps}{\varepsilon}

\newcommand{\CA}{\mathcal{A}}
\newcommand{\CB}{\mathcal{B}}
\newcommand{\CC}{\mathcal{C}}
\newcommand{\CD}{\mathcal{D}}
\newcommand{\CG}{\mathcal{G}}
\newcommand{\CH}{\mathcal{H}}
\newcommand{\CI}{\mathcal{I}}
\newcommand{\CJ}{\mathcal{J}}
\newcommand{\CK}{\mathcal{K}}
\newcommand{\CL}{\mathcal{L}}
\newcommand{\CN}{\mathcal{N}}
\newcommand{\CP}{\mathcal{P}}
\newcommand{\CS}{\mathcal{S}}
\newcommand{\CT}{\mathcal{T}}
\newcommand{\CX}{\mathcal{X}}
\newcommand{\ZZ}{\mathbb{Z}}
\newcommand{\RR}{\mathbb{R}}
\newcommand{\NN}{\mathbb{N}}
\newcommand{\II}{\mathbb{1}}

\newcommand{\va}{\bm{a}}
\newcommand{\vb}{\bm{b}}
\newcommand{\vc}{\bm{c}}
\newcommand{\vd}{\bm{d}}
\newcommand{\ve}{\bm{e}}
\newcommand{\vf}{\bm{f}}
\newcommand{\vg}{\bm{g}}
\newcommand{\vh}{\bm{h}}
\newcommand{\vp}{\bm{p}}
\newcommand{\vr}{\bm{r}}
\newcommand{\vt}{\bm{t}}
\newcommand{\vu}{\bm{u}}
\newcommand{\vv}{\bm{v}}
\newcommand{\vw}{\bm{w}}
\newcommand{\vx}{\bm{x}}
\newcommand{\vy}{\bm{y}}
\newcommand{\vz}{\bm{z}}
\newcommand{\zv}{\bm{0}}
\newcommand{\ov}{\bm{1}}

\newcommand{\vveps}{\bm{\veps}}
\newcommand{\veta}{\bm{\eta}}
\newcommand{\vxi}{\bm{\xi}}
\newcommand{\valpha}{\bm{\alpha}}
\newcommand{\vbeta}{\bm{\beta}}
\newcommand{\vgamma}{\bm{\gamma}}
\newcommand{\vtheta}{\bm{\theta}}
\newcommand{\vlambda}{\bm{\lambda}}
\newcommand{\vnu}{\bm{\nu}}
\newcommand{\vpi}{\bm{\pi}}
\newcommand{\vtau}{\bm{\tau}}

\newcommand{\mA}{\bm{A}}
\newcommand{\mB}{\bm{B}}
\newcommand{\mC}{\bm{C}}
\newcommand{\mD}{\bm{D}}
\newcommand{\mE}{\bm{E}}
\newcommand{\mF}{\bm{F}}
\newcommand{\mG}{\bm{G}}
\newcommand{\mH}{\bm{H}}
\newcommand{\mI}{\bm{I}}
\newcommand{\mL}{\bm{L}}
\newcommand{\mM}{\bm{M}}
\newcommand{\mP}{\bm{P}}
\newcommand{\mQ}{\bm{Q}}
\newcommand{\mR}{\bm{R}}
\newcommand{\mS}{\bm{S}}
\newcommand{\mU}{\bm{U}}
\newcommand{\mV}{\bm{V}}
\newcommand{\mX}{\bm{X}}

\newcommand{\tr}{^{\intercal}}
\newcommand{\ntr}{^{-\intercal}}
\newcommand{\inv}{^{-1}}

\newcommand{\gr}{\mbox{graph}}
\newcommand{\ra}{\rightarrow}
\newcommand{\la}{\leftarrow}
\newcommand{\Ra}{\Rightarrow}
\newcommand{\rra}{\rightrightarrows}
\newcommand{\ptr}{\marginpar{$\Leftarrow$}}

\newcommand{\pfxi}{\frac{\partial f(\vx)}{\partial x_i}}
\newcommand{\pfx}{\partial f(\vx)}
\newcommand{\pf}{\partial f}
\newcommand{\pxi}{\partial x_i}
\newcommand{\px}{\partial x}

\newcommand{\nfx}{\nabla f(\vx)}
\newcommand{\eps}{\epsilon}
\newcommand{\eg}{\textit{e.g.}}
\newcommand{\ie}{\textit{i.e.}}

\newcommand{\vsp}{\vspace{4mm}}
\newcommand{\vspp}{\vspace{8mm}}
\newcommand{\vsppp}{\vspace{12mm}}

\newcommand{\hsp}{\hspace{4mm}}
\newcommand{\hspp}{\hspace{8mm}}
\newcommand{\hsppp}{\hspace{12mm}}

\newcommand{\pr}[1]{\mathbb{P}\left(#1\right)}
\newcommand{\ex}[1]{\mathbb{E}\left[#1\right]}
\newcommand{\variance}[1]{\mbox{Var}\left(#1\right)}
\newcommand{\covar}[1]{\mbox{Cov}\left(#1\right)}
\newcommand{\C}[2]{\left(\begin{array}{c} #1 \\ #2 \end{array}\right)}

\newcommand{\maximize}{\mbox{maximize\hspace{4mm} }}
\newcommand{\minimize}{\mbox{minimize\hspace{4mm} }}
\newcommand{\subto}{\mbox{subject to\hspace{4mm}}}

\newenvironment{sibitemize}{
  \renewcommand{\labelitemi}{$\diamond$}
  \begin{itemize}
    \setlength{\parskip}{0mm}}
  {\end{itemize}}

\newcommand{\propnum}[2]{\vspace{3mm}
  \noindent {\sc Proposition #1}{\it #2} \vspace{3mm}}
\newcommand{\lemnum}[2]{\vspace{3mm}
  \noindent {\sc Lemma #1}{\it #2} \vspace{3mm}}
\newcommand{\thmnum}[2]{\vspace{3mm}
  \noindent {\sc Theorem #1}{\it #2} \vspace{3mm}} 

\maketitle

\begin{abstract}
The concept of counterfactual explanations (CE) has emerged as one of the important concepts to understand the inner workings of complex AI systems. In this paper, we translate the idea of CEs to linear optimization and propose, motivate, and analyze three different types of CEs: strong, weak, and relative. While deriving strong and weak CEs appears to be computationally intractable, we show that calculating relative CEs can be done efficiently. By detecting and exploiting the hidden convex structure of the optimization problem that arises in the latter case, we show that obtaining relative CEs can be done in the same magnitude of time as solving the original linear optimization problem. This is confirmed by an extensive numerical experiment study on the NETLIB library. 

\end{abstract}

\section{Introduction.}

As artificial intelligence (AI) continues to influence our daily lives, the need for interpretability and transparency increases. This need for comprehensive explanations has been accelerated partly by the legislative initiatives such as the General Data Protection Regulation, the European Union AI Act, and the US Blueprint for an AI Bill of Rights \citep{GDPR, EUAIAct, USAIBill}. These regulations emphasize the necessity of providing clear and understandable explanations for automated systems, echoing society's demand for trustworthy AI and aligning with the \emph{right for explanation} principle. 

These developments have attracted the attention of the researchers in machine learning who have started to develop algorithms that pave the way for explainable AI (XAI) \citep{biran2017}. Among these efforts, the concept of counterfactual explanations (CEs) has emerged as one of the key approaches in XAI to understanding the inner workings of complex AI models \citep{wachter2018, maragno2022}. CEs aim to identify the (smallest) change in personal data that would lead to a desired model outcome. A canonical CE example is credit scoring, where a model predicts loan eligibility. If the model denies a loan for an individual, then it should also offer an explanation. For instance a CE might state ``if your annual salary was 1500 EUR higher and your account balance was 900 EUR higher, you would have been granted a loan.''

While much attention is dedicated to the explanations of AI systems, only a few works tackle explainability of decisions stemming from the solutions of optimization problems \citep{korikov2021counterfactualGDPR,korikov2021counterfactual,korikov2023objective,aigner2024framework,goerigk2023framework}. These solutions play a pivotal role in diverse domains, ranging from logistics and finance to healthcare and engineering. Despite their widespread presence, the lack of transparency in the optimal solutions turns these systems into \emph{black boxes}. As a result, the reasoning behind their support in decision-making remains concealed.

In this context, the significance of explanations in optimization becomes apparent as they offer advantages at various levels of application. First, they can be used to support individuals attempting to understand the reasoning behind optimization-driven decisions. Second, stakeholders, such as businesses and public authorities, are impacted by optimization results and can use explanations to get clear justifications for decisions that may have broad implications. Finally, the operations research analyst, who is responsible for setting up complex models, can greatly benefit from substantial insights into the complex interactions of variables. 

To this end, we propose to apply and extend the concept of CEs to linear optimization. As defined in \cite{korikov2021counterfactualGDPR}, we obtain a CE by identifying the (smallest) change needed in the optimization model's parameters such that an optimal solution of the new problem fulfills a set of desired properties. Consider a multi-party resource allocation problem as the counterpart of the credit scoring example above. This problem aims at minimizing the cost under a set of budget constraints. The optimal solution results in allocating only 100 resources to one of the parties. Then, a CE could be ``if the cost of the party decreases to 30 EUR and its budget increases to 2300 EUR, then the allocated amount would have been higher than 110.'' While the latter definition provides a useful concept, it is only one example from a complete set of situations and difficulties appearing in practical applications. First, this definition only considers the existence of an optimal solution with the desired property. Hence, it ignores the fact that multiple optimal solutions may exist and not all of them are desired. Second, the latter concept applied to constraint parameters leads to intractable optimization problems which cannot be solved for realistic instance sizes.


In summary, we make the following contributions to the literature:
\begin{enumerate}
\item We extend the concept of CEs studied in \cite{korikov2021counterfactualGDPR} by defining three different types of CEs which cover more relevant situations occuring in practice.
\item We apply the three concepts to linear optimization problems where both the objective and constraint parameters can be adjusted.
\item We derive formulations to calculate all three types of CEs and analyze the mathematical structure of the proposed formulations, identify their challenges, and propose approaches to solve them.
\item We computationally test all three types of CEs on a diet problem based on real-world data. Furthermore, we conduct an extensive computational study for the relative CEs on linear optimization problems from the NETLIB library. To make sure that our results are reproducible, we provide a dedicated repository to replicate our experiments\footnote{https://github.com/JannisKu/CE4LOPT}.
\end{enumerate}

\subsection{Examples.}\label{sec:examples}
In this section, we describe in detail several motivating examples where counterfactual explanations are crucial. These examples are taken from real life applications, and many of them were executed by the authors. 

\noindent {\bf Diet problem.}
In \cite{Peters2016,peters2022world} a linear optimization model is developed to optimize the food supply chain for the United Nations' World Food Programme (WFP). This model has been and is used for each project of the WFP, and has enabled WFP to feed millions of people. An important part of the model is the diet problem: Which food commodities should be included in the food basket such that all nutritional requirements are satisfied, while the costs are minimized? The food commodities can be purchased from many different suppliers. Suppose that the optimal solution of the linear optimization model is such that a certain potential supplier of a certain food commodity is not part of the solution. It is certainly not enough to tell this supplier that the model outcome is such that nothing will be purchased. The supplier wants to know for which reduction in the purchasing cost, her commodity is part of an optimal food basket. In case she is able to change the nutritional contents of her food commodity, the question is: ``What is the minimal change in these nutritional contents such that her commodity is part of the food basket?'' Providing such insights to the suppliers is also beneficial for WFP, since suppliers might change the purchasing costs or nutritional contents.

A similar situation occurs in the mobile application Feed Calculator \citep{FeedCalculator}, which is now used by thousands of small farmers in Africa and Asia to optimize the ingredients for the cattle feed. Each possible ingredient can be purchased at a certain local supplier. 
The core of this application is the diet model, which is a small linear optimization problem. The local suppliers could use counterfactual explanation for this linear optimization model to detect minimal changes to the costs or nutritional contents such that its food commodity becomes attractive for being purchased by local farmers.

We use the diet problem in this work as a running example to explain the different concepts.

\noindent {\bf Facility location problem.}
One of the Sustainable Development Goals of the United Nations is good geospatial accessibility to healthcare centers in low- and middle-income countries. Facility location approaches have been developed to optimize geospatial accessibility given a certain budget to build new centers \citep{TLVN2024}. For local governments it is crucial to know, for example, why in the optimal solution there is no center opened in their districts. Counterfactual questions as ``What is the minimal change in the budget, or in the costs for building a center in their district, or in the population density, that leads to an optimal solution where a center is chosen in their district?''

Of course similar counterfactual explanations are needed in other classical facility problems. For example, one of the authors optimized the physical distribution structure for Philips in Europe. Several distribution centers were closed, and new ones were opened. Of course, the management of these centers that were going to be closed has to be explained why a closure is necessary. Counterfactual explanations are the ideal tool for that. 

\noindent {\bf Network flow problem.}
Many supply chain problems can be modeled as (multi-commodity) network problems. For example, in the linear optimization model for WFP's food supply chain \citep{Peters2016,peters2022world}, the diet problem is integrated into a network model to model the multi-modal transportation from the supplier to the beneficiaries. Suppose that a certain port is not used according to the optimal solution of the linear optimization model. The authorities of this port would like to know what is the minimal reduction in costs such that it is used in the optimal solution. Similar counterfactual explanations are needed for potential transportation companies. 

\noindent {\bf Flood safety problem.}
Optimization has been used to determine new safety standards for the dike heights in the Netherlands \citep{flood2,flood1}. For each of the 53 so-called dike-ring areas, i.e., an area that is protected by dikes, finally one out of five different safety levels has been chosen. For the management of a specific dike-ring area it is crucial to know what minimal change in the dike-ring area characteristics (as for example the number of people or the economic value in that area) would have led to a higher safety level.  

Most of the above examples are at the tactical or even strategical level. Indeed, especially for strategical decisions that affect multi-stakeholders, counterfactual explanations are needed. However, we emphasize that such explanations are also often needed in more operational decisions, where it can affect the personal lives. Here are two examples taken from the Franz Edelman Award finalists and winners.

\medskip
\noindent {\bf Scheduling trains in the Netherlands.}
In December 2006, Netherlands Railways introduced a completely new timetable by using sophisticated operations research techniques \citep{NS}. The first versions of the schedule led to much social unrest for the employees. The schedules were considered as ``boring,'' since for each employee it was more or less the same every day. After these aspects were included in the approach, and the resulting schedule was better explained, the employees finally accepted the new schedule. 

\medskip
\noindent {\bf Boston public school transport.}
\cite{Boston} describe that optimization methods were developed for Boston Public Schools (BPS) to create a better way to construct bus routes. The goals were improving efficiency, deepening the ability to model policy changes, and realigning school start times.  Using this methodology, BPS proposed a solution that would have saved an additional \$12 million annually and also shifted students to more developmentally appropriate school start times (\textit{e.g.}, by reducing the number of high school students starting before 8:00 a.m. from 74\% to 6\% and the average number of elementary school students dismissed after 4:00 p.m. from 33\% to 15\%). However, 85\% of the schools’ start times would have been changed, with a median change of one hour. This magnitude of change led to strong vocal opposition from some school communities that would have been affected negatively. Therefore, BPS did not implement the plan.

In the examples above, multiple stakeholders are involved, and these stakeholders need counterfactual explanation not only for the sake of explanation, but also for actually changing the input parameters. Counterfactual explanations can thus also be used as a tool for negotiation among multiple stakeholders. 

Counterfactual explanations could also be valuable in single stakeholder environments. Minimal changes in input parameters such that a certain solution is optimal gives much insight in the decision problem. Feasibility ensuing counterfactual explanations are specifically very valuable for answering questions like ``What is the minimal change in the input data such that the problem becomes feasible?'' 


Notice that in all the examples mentioned above, even for the simplest diet problem, the {\em factual} explanations do not work. The linear optimization model and the simplex or interior point methods are too difficult to explain to a non-expert. However, ``what-if scenario analysis'' could provide partial explanation in some cases. For example, enforcing in the facility location model that a certain center is opened, one could optimize for the overall accessibility, and then calculate the accessibility decrease. The explanation is then: ``If we open this center, then there will be a reduction of $x$\% in overall accessibility. Since in general obtaining explanations by ``what-if scenario analysis'' is computationally much easier than by counterfactual explanations, we advocate to use both ways. 

We finally point out that the result of counterfactual explanation could also be that the minimal change is extremely small, or even no change has to be performed. The last case could happen when there are multiple optimal solutions and at least one of it already has the desired properties. For those cases, we argue that {\em factual} explanations should be added, based on secondary criteria not included in the linear optimization model.   



\subsection{Related Work.}

We first discuss the most related works for explainability in optimization.

\cite{korikov2021counterfactualGDPR} and \cite{korikov2023objective,korikov2021counterfactual} study a definition of counterfactual explanations for integer optimization problems. This definition coincides with our weak CE definition; see Section \ref{sec:definitions}. \cite{korikov2021counterfactualGDPR} introduce the concept and study it for the case where only the objective function parameters of the problem may be adjusted. Furthermore, they restrict their approach to the case that the desired solution property may only be defined on a single variable. Additionally, they assume that no such desired solution is optimal for the present problem. We relax all these assumptions in our work. \cite{korikov2021counterfactual} connect the idea of CEs to inverse optimization and use inverse constraint programming to solve the problem where again only the objective function parameters may be adjusted. Finally, \cite{korikov2023objective} generalize the latter works to the case that the constraints for the desired outcome can be defined on all variables. They develop a constraint generation algorithm which can solve the CEs problem for discrete optimization problems if only the objective parameters may be adjusted.

\cite{forel2023explainable} assume that the parameters of the present optimization problem are derived by predictions based on additional context parameters (\textit{e.g.}, the weather, day or temperature). They consider the smart-predict-then-optimize method of \cite{bertsimas2020predictive}, where random forests and nearest neighbor predictions are used. They adapt the idea of CEs to calculate CEs in the context parameter space where changes are only considered in the objective parameters of the problem.

A different concept for explaining optimization problems, which is not connected to counterfactual explanations, was developed by \cite{aigner2024framework}. The authors present a data-driven explanation term, based on historical solutions for the same problem class, which is added to the objective function of the optimization problem to increase the explainability of the solution. 

Finally, \cite{goerigk2023framework}  propose an inherently interpretable model, which provides interpretability of the derived solutions. The authors calculate a decision tree and a small set of solution such that each future problem instance is mapped by the tree to one of the determined solutions. While this approach can be generalized to constraint parameters, it is mainly studied for the objective function parameters.

We next discuss how some of the known techniques in optimization are related to CEs in linear optimization.

\cite{Dantzig} develops general linear programming, in which (some of) the parameters of the linear optimization problem are variables too.  This method is extensively extended to a much wider class of problems \citep{gorissen2024}. We will use these techniques to calculate CEs for linear optimization. 

There is a strong relationship with inverse optimization, 
in which the optimal solution is given, and the aim is to calculate unknown parameters in the objective or constraints. 
For a good survey on inverse optimization, we refer the reader to the recent work of \cite{chan2021inverse}.
Calculating CEs for linear optimization appears to be more difficult, since the precise optimal solution is not given, but only some constraints that the (optimal) solution has to satisfy. 
There are several papers on partial inverse optimization, but in those papers only objective parameters can change and finally a computationally intractable bilinear problem has to be solved; see for instance \citep{Wang2013}.


Several papers (e.g., \cite{amaral2008reformulation}, \cite{barratt2021automatic}, and \cite{moosaei2021optimal}) study  the question ``What is the minimal change in the input data such that the problem becomes feasible?'' This is a simple version of the counterfactual explanation concept. Especially the computationally tractable methods for relative counterfactuals developed in this paper could also be used to answer such feasibility questions.

Parametric optimization and sensitivity analysis are related concepts where changes in single parameters (leading to changes in multiple coefficients in the former case) are studied. However, the goal of these concepts is to analyze the region of changes which can be performed without changing the optimal solution (or optimal basis).  Moreover, sensitivity analysis can be seen as factual explanation: it analyzes the effect on the optimal solution when we change the values of the problem parameters.

\section{Definitions for Counterfactual Explanations.}\label{sec:definitions}
In this work, we propose to translate the concept of CEs in machine learning to linear optimization. Consider an instance of a linear programming (LP) problem of the form
\begin{equation}
\label{eq:LP}
\begin{aligned}
\min \ & c^\top x \\
s.t. \quad & Ax\ge b,\\
& x\ge 0,
\end{aligned}
\end{equation}
where $c\in \R^n$ and $A\in \R^{m\times n}$ and $ b\in\R^m$. This problem can be represented by its corresponding problem parameters $(c,A,b)$. An optimization algorithm for such an LP can be interpreted as a function which maps every instance $(c,A,b)\in\mathcal H$ to an optimal solution $x^*$ of the corresponding LP. For a given \textit{factual instance} $(\hat c, \hat A, \hat b)\in\mathcal H$, a counterfactual explanation is a --preferably similar-- instance $(c,A,b)\in \mathcal H$ for which the optimal solution lies in a \textit{favored solution space}, $\mathcal D(\hat x)$. This space does not contain the optimal solution $\hat x$ of the factual instance. In other words, a CE is an update in the optimization parameters such that the optimal solution of the LP with the updated parameters has a given list of desired properties. The main concept described above is visualized in Figure \ref{fig:CE_for_opt}. 

Note that in general more than one solution can be optimal for an LP and hence the calculated optimal solution may depend on the choice of the optimization algorithm. There can also be solutions, which are only feasible for the LP, but they return smaller objective function values than a fixed target. These observations give rise to extend the current framework for three different types of CEs, namely \textit{weak counterfactual explanations}, \textit{strong counterfactual explanations}, and \textit{relative counterfactual explanations}. 

\begin{figure}[h!]
    \centering
    \includegraphics[scale=0.3]{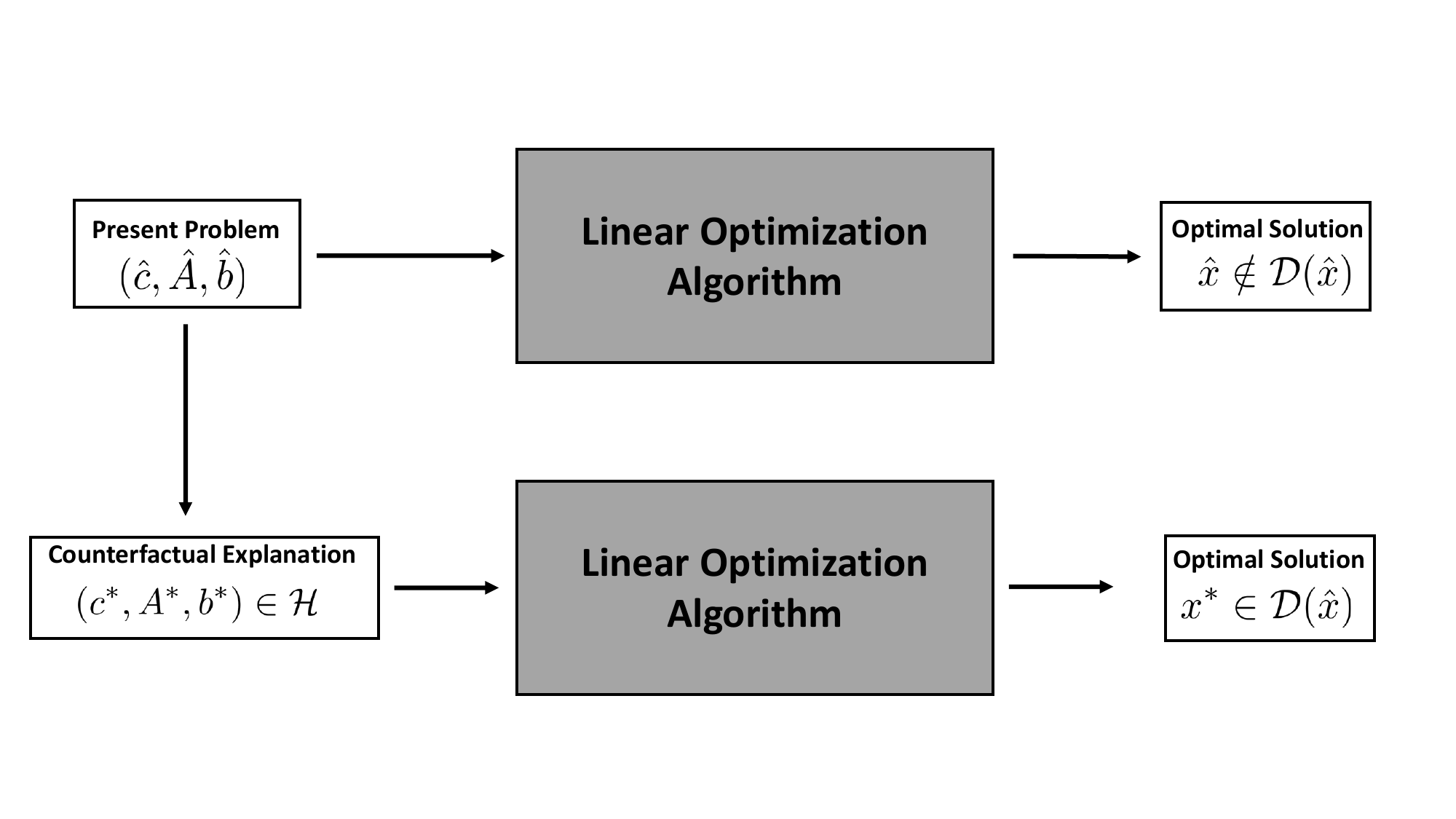}
    \caption{An overview over the concept of counterfactual explanations for linear optimization problems.}
    \label{fig:CE_for_opt}
\end{figure}

Assume the \textit{present problem} is given as
\begin{equation}\tag{PP}
\label{eq:presentProblem}
\begin{aligned}
\min \ & \hat c^\top x \\
s.t. \quad & \hat Ax\ge \hat b,\\
& x\ge 0,
\end{aligned}
\end{equation}
where $\hat c\in \R^n$ and $\hat A\in \R^{m\times n}$ and $\hat b\in\R^m$. The present problem can be interpreted as the optimization problem which was solved by the \textit{decision maker} to obtain an optimal decision $\hat x$. Each variable can be interpreted as being related to what we call a \textit{stakeholder}. In the diet problem example, the decision maker is the organization (including operations research analaysts) that has control over the data as well as the optimization problem, and that is responsible for implementing the final decision. Each variable corresponds to a product sold by a food supplier, which is a stakeholder.

Assume that a given subset of problem parameters is mutable and can be changed within a certain feasible region. To this end, we define the \textit{mutable parameter space} $\mathcal H\subseteq \R^n\times \R^{m\times n} \times \R^m$ which contains the parameters of the present model, \textit{i.e.}, $(\hat c,\hat A,\hat b)\in \mathcal H$. The mutable parameter space can be defined by intervals for each parameter in which the parameter changes or even by a polyhedral or ellipsoidal structure if the allowed changes of parameters depend on each other. For example, if a stakeholder, influenced by the optimization process, has access to change the parameters of the $i$-th column of $A$ (denoted as $A_i$) and the $i$-th cost parameter $c_i$, the mutable parameter space is defined as
\[
\mathcal H=\left\{ (c,A,b): c_j=\hat c_j \text{ and } A_j=\hat A_j, \forall j\neq i, \ b=\hat b, \ c_i\in [\underline c_i, \bar c_i], \ A_i\in [\underline A_i, \bar A_i] \ \right\},
\]
where $\underline c_i, \bar c_i, \underline A_i, \bar A_i$ are the corresponding upper and lower bounds on the mutable parameters. In the diet problem the objective parameter $c_i$ corresponds to the price of a product and the column $A_i$ corresponds to the different nutrient values of the product. For every parameter the corresponding stakeholder can define her mutable space depending on how much she is able to change the parameters.

Finally, we define the \textit{favored solution space} $\mathcal D(\hat x)$ which is the set of solutions $x$ which are favored by a stakeholder which is influenced by the decision $\hat x$. For example, in the diet problem the favored solution space could be the set of solutions, for which the stakeholder sells at least $5\%$ more of her product $i$, i.e.,
\[
\mathcal D(\hat x)=\left\{ x\in\mathcal X: \ x_i\ge 1.05\hat x_i\right\}.
\]
Note that the set $\mathcal D(\hat x)$ does not necessarily has to overlap with the feasible region of the present problem. Furthermore, the favored solution space can be independent of $\hat x$ and we denote it as $\mathcal D$ in this case.

\subsection{Weak Counterfactual Explanations.} The stakeholder which is influenced by the decision $\hat x$ can ask the following counterfactual question: 

\begin{quote}
    ``\textit{What is the minimal change of the mutable parameters I have to make such that a solution from the favored solution space is optimal?}''
\end{quote}

This leads to the following more formal definition.

\begin{definition}[Weak Counterfactual Explanation]
A weak counterfactual explanation is a point $(c,A,b)\in \mathcal H$ such that there exists an optimal solution $x^*$ of Problem \eqref{eq:LP} which lies in the \textit{favored solution space} $\mathcal D(\hat x)$.
\end{definition}
In other words, a weak counterfactual explanation is an update of the mutable parameters of the present problem such that at least one optimal solution exists which belongs to the favored solution space. However, since multiple solutions can be optimal, there is no guarantee that the optimal solution implemented by the decision maker is contained in the favored solution space. Note that usually the stakeholder is interested in the smallest changes of the mutable parameters such that the latter definition holds. More formally, the stakeholder is looking for a weak CE $(c,A,b)\in \mathcal H$ which is closest to the point $(\hat c,\hat A,\hat b)$ in a distance metric $\delta: \mathcal H \times \mathcal H \to \R_+$.

In the following proposition we show a property of weak CEs which was already used in \cite{korikov2021counterfactualGDPR} do define the concept of CEs.
\begin{proposition}\label{prop:weak_CEs}
    If $(c,A,b)$ is a weak CE then
    \begin{equation*}
\begin{aligned}
    \min \ & c^\top x \\
    s.t. \quad & Ax\ge b \\
    & x\ge 0
\end{aligned}
\quad = \quad
\begin{aligned}
    \min \ & c^\top x \\
    s.t. \quad & Ax\ge b \\
    & x\in \mathcal D(\hat x) \\
    & x\ge 0 ,
\end{aligned}
    \end{equation*}
i.e., the constraints in $\mathcal D(\hat x)$ are redundant.
\end{proposition}
In fact the equation in the proposition is used in \cite{korikov2021counterfactualGDPR} to describe counterfactual explanations, which shows that indeed the authors study weak counterfactual explanations.





\subsection{Strong Counterfactual Explanations.}
It is important to note that in the definition of weak CEs, we only require that there \textit{exists} an optimal solution of the new optimization problem which fulfills the requirements in $\mathcal D(\hat x)$. However, depending on the optimization algorithm the decision maker uses, this solution may not be implemented, since there might be alternative optimal solutions. To avoid this issue, we define the strong version of counterfactual explanations. Here, the stakeholder which is influenced by the decision $\hat x$ asks the following counterfactual question: 

\begin{quote}
    ``\textit{What is the minimal change of the mutable parameters I have to make, such that all optimal solutions are contained in the favored solution space?}''
\end{quote}

This leads to the following definition.

\begin{definition}[Strong Counterfactual Explanation]
A strong counterfactual explanation is a point $(c,A,b)\in \mathcal H$ such that the set of optimal solutions $\mathcal X^*$ of problem \eqref{eq:LP} lies in the \textit{favored solution space} $\mathcal D(\hat x)$, i.e., $\mathcal X^*\subset \mathcal D(\hat x)$.
\end{definition}

The difference here to the weak version is that we require {\em all} optimal solutions of the new optimization problem to fulfill the requirements in $\mathcal D(\hat x)$. This is an important difference, since it would guarantee that a decision from the favored solution space will be implemented by the decision maker, independently of the solution method used to determine the optimal decision.

Note that usually the stakeholder is interested in the smallest changes of the mutable parameters such that the latter definition holds. More formally, the stakeholder is looking for a strong CE $(c,A,b)\in \mathcal H$ which is closest to the point $(\hat c,\hat A,\hat b)$ in a distance metric $\delta: \mathcal H \times \mathcal H \to \R_+$.

\begin{proposition}
    From the definitions of weak and strong CEs it follows:
    \begin{enumerate}
    \item Every strong CE is also a weak CE.
        \item If the set of optimal solutions for a weak counterfactual explanation $(c,A,b)$ is a singleton, then $(c,A,b)$ is a strong CE as well.
    \end{enumerate}
\end{proposition}

\subsection{Relative Counterfactual Explanations.}
In the following we introduce a third definition for counterfactual explanations. In this setup, the stakeholder is not necessarily interested in the parameter changes which lead to an optimal solution contained in the favored solution space, but only in the parameter changes which lead to a favored solution which has a similar objective value as the present problem. Here the stakeholder which is influenced by the decision $\hat x$ asks the following counterfactual question: 

\begin{quote}
    ``\textit{What is the minimal change of the mutable parameters I have to make such that a solution from the favored solution space changes the objective value at most by a fixed factor?}''
\end{quote}

In the following we assume that for the optimal value of the present problem (PP) it holds that $\hat c^\top \hat x\ge 0$.

\begin{definition}[Relative Counterfactual Explanation]
For a given factor $\alpha\in [0,\infty)$ a relative counterfactual explanation is a point $(c,A,b)\in \mathcal H$ such that there exists a feasible solution in 
\[\left\{ x\in \mathcal X: Ax\ge b, \ c^\top x \le \alpha\hat c^\top \hat x\right\}\cap \mathcal D(\hat x) .
\]
\end{definition}
In contrast to weak and strong CEs this definition does not require optimality of a solution $x$ in the favored solution space, but requires only feasibility instead. Note that in case the factor $\alpha$ is smaller than one, we are aiming for an improvement of the optimal objective function value, while for $\alpha \ge 1$ a certain deterioration of the objective function value is accepted.

Note that usually the stakeholder is interested in the smallest changes of the mutable parameters such that the latter definition holds. More formally the stakeholder is looking for a relative CE $(c,A,b)\in \mathcal H$ which is closest to the point $(\hat c,\hat A,\hat b)$ in a distance metric $\delta: \mathcal H \times \mathcal H \to \R_+$.

\begin{proposition}
From the definitions of weak and relative CEs it follows:
    \begin{enumerate}
        \item If $\alpha=1$ and the parameters of the present problem $(\hat c, \hat A, \hat b)$ is a relative CE, then $(\hat c, \hat A, \hat b)$ is a weak CE as well.
        \item For every weak CE there exists an $\alpha$ such that the same point is also a relative CE.
    \end{enumerate}
\end{proposition}
Note that in case (i) of the latter proposition, since $\hat x\notin \mathcal D(\hat x)$, this means that the problem has multiple optimal soluions and at least one of it lies in $\mathcal D(\hat x)$.

\subsection{Practical Implications.}
In this section, we discuss the different ways of how the different counterfactual explanation concepts presented in the previous subsections can be used in practice. One of the main questions in each application is: ``Who is providing the counterfactual explanation to whom?''

\paragraph{Decision maker provides explanations to stakeholders.} A frequently occuring situation in practice is that a decision maker who has ownership on the data and the optimization problem is calculating the solution $\hat x$ which is afterwards implemented while the involved stakeholders (suppliers or workers) do not have access to the optimization problem. This is the case in all the examples presented in Section \ref{sec:examples}. For example in the diet problem, the decision maker is calculating an optimal solution of the corresponding linear optimization problem to decide how much of each product is bought from which supplier. In this case, a supplier asks the decision maker what would be the minimal change in prices she has to perform such that the decision made by the decision maker would be to buy at least a certain amount of a product from her. If she asks for a strong CE the decision maker will return the minimal change in prices she has to perform such that for any optimal solution (independent of the solution algorithm) the required amount will be purchased from her. In contrast, if she asks for a weak CE the decision maker can return the minimal changes in prices together with the information how much of the product would be purchased from her after the change of the prices. In this case it could be that the final decision is not meeting her requirements since not in every optimal solution the requirements have to be fulfilled for a weak CE. Additionally, if the decision maker changes her solution algorithm (or its settings) in the future it can be that the final decision will be different. In case of a relative CE the supplier has to provide the factor $\alpha$ and the decision maker will return again the minimum change in prices. However, here the decision maker has to decide if an increase/decrease of the costs by a factor of $\alpha$ is desirable and if a corresponding solution can be implemented.

\paragraph{Analysis of the problem.} Especially the concept of relative CEs can be used to perform an analysis of the problem which is faced by a decision maker. In the facility location problem to optimize geospatial accessibility to supplies in case of a disaster the decision maker maybe wants to analyze the effects of parameter changes on the rising costs. For example, an increase of the population in a certain region can effect the costs. In case of a what-if analysis, the decision maker can answer questions of the form: if the population in every region increases by $2\%$ what will be my optimal costs? However, by solving the relative CE problem, we can answer questions now of the form: ``What is the minimal increase of the population such that my costs will increase by at most $3\%$?'' 

\paragraph{Feasibility analysis.}
The concept of CEs is also beneficial in the situation where the present problem is infeasible and the decision maker wants to find the smallest changes in the parameter space which would lead to a feasible problem. In this case the weak CE problem can be solved with $\mathcal D(\hat x)=\R^n$. Consider for example a network flow problem where the capacities are given and the different involved stakeholders define the demand and supply for each node. It could be that for the given situation there is no feasible flow. The decision maker can now ask the questions: ``What is the minimum change in the capacity parameters that I have to perform such that there exists a feasible flow?''
In this situation, we could also apply relative CEs without the constraint for the objective value. Furthermore, the concept of strong CEs is equivalent to weak CEs in this example.

\section{Calculating Counterfactual Explanations.}
In this section, we present optimization models which can be used to calculate all three variants of counterfactual explanations presented in Section \ref{sec:definitions}. In the following we assume that the present problem is feasible.

\subsection{Weak Counterfactual Explanations.}\label{sec:results_weakCE}
The weak counterfactual explanation problem is defined as
\begin{subequations}
\label{eq:WCEP} 
\begin{align}
(\text{WCEP}):\quad \inf_{x,c,A,b} \ & \delta\left( (c,A,b), (\hat c, \hat A, \hat b)\right) \label{eq:weak_objective}\\
s.t. \quad & x\in \argmin_{z: Az\ge b, z\ge 0} c^\top z, \label{eq:weak_optimality_condition}\\
& x\in D(\hat x), \label{eq:weak_favored_sol_space}\\
&(c,A,b)\in \mathcal H, \label{eq:weak_mutable_space}
\end{align}
\end{subequations}
where $\delta: \mathcal H \times \mathcal H\to \R_+$ is a given distance function in the parameter space. Note that in the objective function \eqref{eq:weak_objective}, we minimize the distance to the parameters of the present problem. Constraint \eqref{eq:weak_optimality_condition} imposes that each feasible $x$ is an optimal solution for the problem with chosen parameters $(c,A,b)$. Constraint \eqref{eq:weak_favored_sol_space} ensures that  the optimal solution $x$ lies in the favored solution space and constraint \eqref{eq:weak_mutable_space} implies that we only consider parameter changes inside the mutable parameter space $\mathcal H$. The WCEP can be interpreted as an optimistic bilevel problem; see e.g. \cite{dempe2015bilevel,kleinert2021survey}.

In the following theorem, we present a reformulation of the latter problem as a bilinear problem.

\begin{theorem}\label{thm:weakCE_reformulation}
Problem WCEP is equivalent to the following problem:
\begin{subequations}
\begin{align}
(\text{WCEP'}):\quad \inf_{x,y,c,A,b} \ & \delta\left( (c,A,b), (\hat c, \hat A, \hat b)\right) \label{eq:weak_reformulation_objective}\\
s.t. \quad & c^\top x \le b^\top y, \label{eq:weak_strong_duality}\\
& A^\top y \le c, \label{eq:weak_dual_feasible}\\
&Ax\ge b, \label{eq:weak_primal_feasible}\\
&x\in D(\hat x),\\
&(c,A,b)\in \mathcal H, \\
& x,y\ge 0. \label{eq:weak_greater_zero_variables}
\end{align}
\end{subequations}
\end{theorem}
\begin{proof}
We replace the optimality constraint \eqref{eq:weak_optimality_condition} by constraints \eqref{eq:weak_strong_duality}--\eqref{eq:weak_primal_feasible} and \eqref{eq:weak_greater_zero_variables}, where \eqref{eq:weak_primal_feasible} and \eqref{eq:weak_greater_zero_variables} ensure primal feasibility for $x$, \eqref{eq:weak_dual_feasible} and \eqref{eq:weak_greater_zero_variables} ensure dual feasibility for $y$, and \eqref{eq:weak_strong_duality} ensures that primal and dual objective function values are equal. By the classical strong duality result, it follows that any feasible $x$ must be optimal for the minimization problem in \eqref{eq:weak_optimality_condition}.
\end{proof}

\noindent We next show the following {\em structural} properties of the WCEP:
\begin{enumerate}
    \item The projection of the feasible region on the $(c,A,b)$-space can be open.
    \item The projection of the feasible region on the $(c,A,b)$-space may be non-convex and disconnected, even if $\mathcal H$ and $\mathcal D(\hat x)$ are convex sets, and even if only the objective parameters in $c$ are allowed to change (see Example \ref{ex:feasible_region_non-convex} and \ref{ex:feasible_region_disconnected}).
    \item The projection of the feasible region on the $x$-space can be convex if the $0$-cost vector is contained in $\mathcal H$ and $\mathcal D(\hat x)$ is convex.
    \item The projection of the feasible region on the $x$-space (and $(x,y)$-space) can be non-convex (see Example \ref{ex:nonconvex_region_x}).
\end{enumerate}

\noindent We start with showing that the feasible region of the $(c,A,b)$-space can be open. Due to this observation we cannot guarantee that an optimal solution of the WCEP exists, which is why we use the infimum instead of the minimum operator.
\begin{example}\label{thm:openess_weakCE}
Consider the present problem 
\begin{align*}
    \min \ & x \\
    s.t. \quad & \hat ax=\hat a \\
    x\ge 0
\end{align*}
where $\hat a=0$. The unique optimal solution is $\hat x = 0$. Assume the favored solution space $\mathcal D(\hat x) = \{ x\in\R: x\ge 1\}$ and $\mathcal H = [0,1]$. For every $a>0$ the optimal solution of the problem is $x=1$ which lies in $\mathcal D(\hat x)$. Hence, the feasible region for parameter $a$ in WCEP is $(0,1]$ which is open. In this case the WCEP does not have an optimal solution. 
\end{example}




The following example shows that the feasible region projected on the $(c,A,b)$-variables can be non-convex.

\begin{example}\label{ex:feasible_region_non-convex}
The following example shows that the projection of the feasible region of WCEP on the variable space $(c,A,b)$ may be non-convex; see Figure \ref{fig:example_non-convex} for an illustration. Consider the present problem 
\begin{align*}
\max \ & x_1 \\
s.t. \quad & |x_1|+|x_2|\le 1,
\end{align*}
which can be reformulated into a linear optimization problem. Note that we do not require here that $x\ge 0$ as we do in Theorem \ref{thm:weakCE_reformulation}. However, this can be easily achieved by shifting the feasible region into the positive orthant, e.g., by replacing $x_i$ by $x_i-1$ everywhere in the problem. The problem has the unique optimal solution $\hat x = (1,0)$. Suppose now that the favored solution space is given by 
\[
\mathcal D=\left\{ x\in\R^2: \ -0.5\le x_1 \le 0.5\right\}.
\]
Furthermore, we assume that only the two objective parameters $c_1,c_2$ can be changed to any value. Let int$(\mathcal S)$ denote the interior of a set $\mathcal S$. Then, the following statements hold:
\begin{itemize}
    \item For any $(c_1,c_2)\in \text{int } (\mathcal C_1)$ where $\mathcal C_1=\left\{(c_1,c_2)= \lambda_1\binom{1}{1} + \lambda_2\binom{1}{-1}, \lambda_1,\lambda_2\ge 0\right\}$
    the unique optimal solution of the present problem is $x^1=(1,0)\notin \mathcal D$, and hence, $\mathcal C_1$ does not contain any feasible weak CE.
    \item For any $(c_1,c_2)\in \text{int }(\mathcal C_2)$ where $\mathcal C_2=\left\{(c_1,c_2)= \lambda_1\binom{-1}{1} + \lambda_2\binom{-1}{-1}, \lambda_1,\lambda_2\ge 0\right\}$
    the unique optimal solution of the present problem is $x^2=(-1,0)\notin \mathcal D$ and hence $\mathcal C_2$ does not contain any feasible weak CE.
    \item For any $(c_1,c_2)\in \mathcal C_3$ where $\mathcal C_3=\left\{(c_1,c_2)= \lambda_1\binom{1}{1} + \lambda_2\binom{-1}{1}, \lambda_1,\lambda_2\ge 0\right\}$
    the point $x^3=(0,1)\in\mathcal D$ is optimal, and hence, all points in $\mathcal C_3$ are weak CEs.
    \item For any $(c_1,c_2)\in \mathcal C_4$ where $\mathcal C_4=\left\{(c_1,c_2)= \lambda_1\binom{1}{-1} + \lambda_2\binom{-1}{-1}, \lambda_1,\lambda_2\ge 0\right\}$
    the point $x^4=(0,-1)\in \mathcal D$ is optimal, and hence, all points in $\mathcal C_4$ are weak CEs.
\end{itemize}
Note that $\text{int } (\mathcal C_1)\cup \text{int }(\mathcal C_2)\cup  \mathcal C_3 \cup \mathcal C_4=\mathbb R^2$. Thus, the region of feasible weak counterfactuals is $\mathcal C_3 \cup \mathcal C_4$, which is non-convex.
\end{example}
\begin{figure}
    \centering
    \includegraphics[scale=0.3]{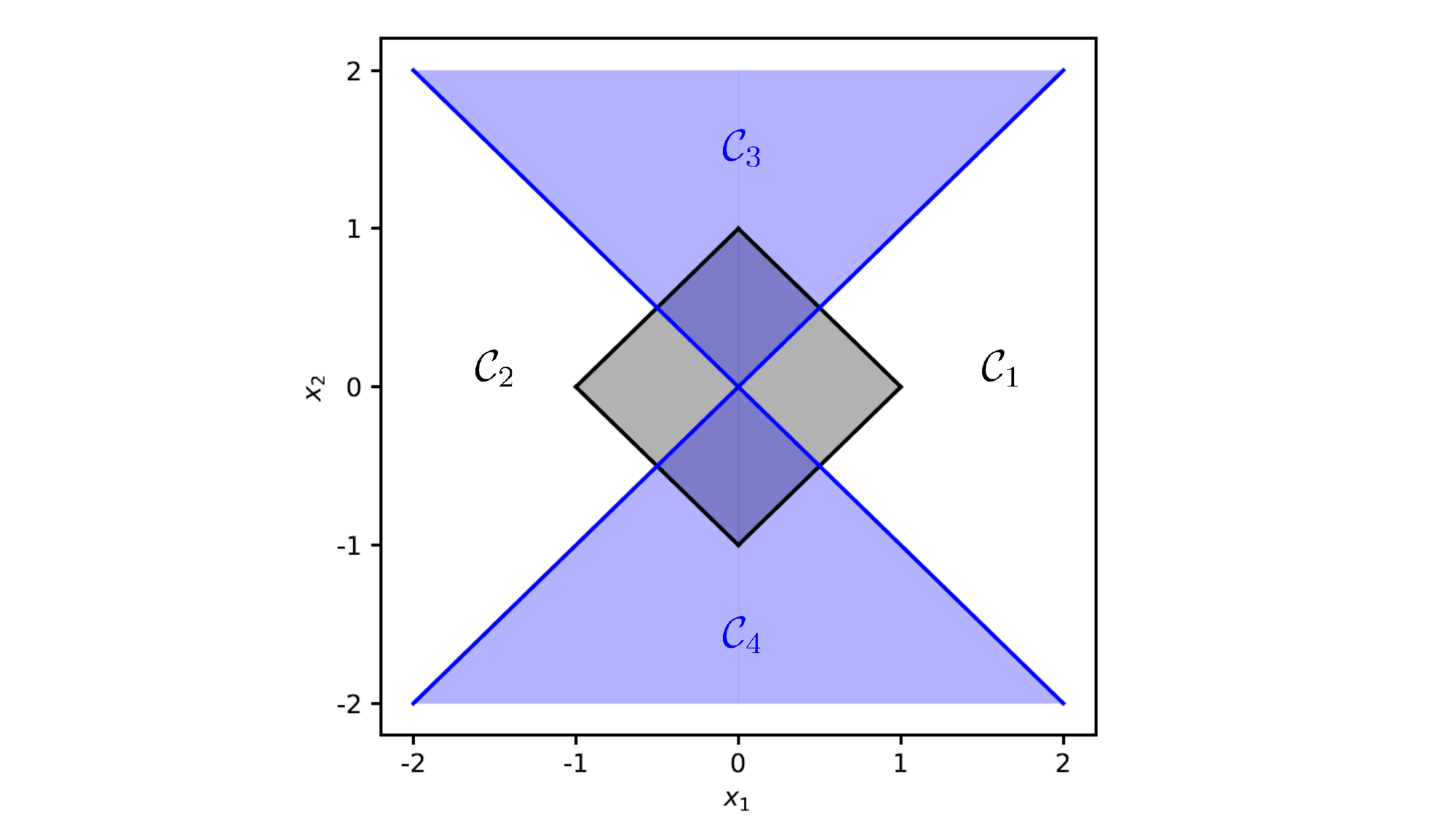}
    \caption{The feasible region of the present problem of Example \ref{ex:feasible_region_non-convex} (in grey) and the feasible region of the weak CEs for the objective parameters (in blue).}
    \label{fig:example_non-convex}
\end{figure}


\begin{example}\label{ex:feasible_region_disconnected}
The following example shows that the projection on the feasible region of the variables $c,A,b$ may be disconnected. Consider the same setup as in Example \ref{ex:feasible_region_non-convex} and additionally define the mutable parameter space $\mathcal H=\{(c_1,c_2): 0.5\le c_1,c_2 \le 1.5\}$. Then the feasible region of the weak CEs is given as $(\mathcal C_3\cup \mathcal C_4)\cap \mathcal H$ which is a disconnected set.
\end{example}

\begin{proposition}
If $(0,\hat A,\hat b)\in \mathcal H$, then the projection of the feasible region of Problem WCEP onto the $x$-space is convex. 
\end{proposition}
\begin{proof}
For the parameters $(0,\hat A,\hat b)$ all feasible points of the present problem are optimal. Hence, the projection onto the $x$-space is the feasible set of the present problem intersected with $\mathcal D(\hat x)$, which is convex.
\end{proof}
\begin{example}\label{ex:nonconvex_region_x}
The following example shows that the projection of the feasible region of WCEP on the $x$-space may be non-convex. 
Consider the linear optimization problem 
\begin{align*}
\max \ & c_1x_1 + c_2x_2 \\
s.t. \quad & x_1 + x_2 \le 2, \\
& x_1 - x_2 \le 0, \\
& x_1,x_2\ge 0, 
\end{align*}
where the feasible region is shown in Figure \ref{fig:example_non-convex_on_x}. We assume that only the cost parameters can change, where $(c_1,c_2)\in \{1\}\times [-1,1]$ and set $\mathcal D=\R_+^2$. Note that the point $z_1=(0,2)$ is feasible and optimal for cost vector $c=(1,1)$. On the other hand, the point $z_2=(0,0)$ is feasible and optimal for cost vector $c'=(1,-1)$. However, for the point $\tilde z = \frac{1}{2} z_1 + \frac{1}{2}z_2 = (0,1)$ there exists no cost vector in $\{1\}\times [-1,1]$ for which it is optimal. The reason is that for $c_2>0$, the point $(0,2)$ always has a strictly better objective function value, while for $c_2<0$, the point $(0,0)$ always has a strictly better objective function value. For $c_2=0$, the point $(1,1)$ has a strictly better objective function value, since $c_1=1$. Hence, the projection on the $x$-space is not convex. 
\end{example}

\begin{figure}
    \centering
    \includegraphics[scale=0.3]{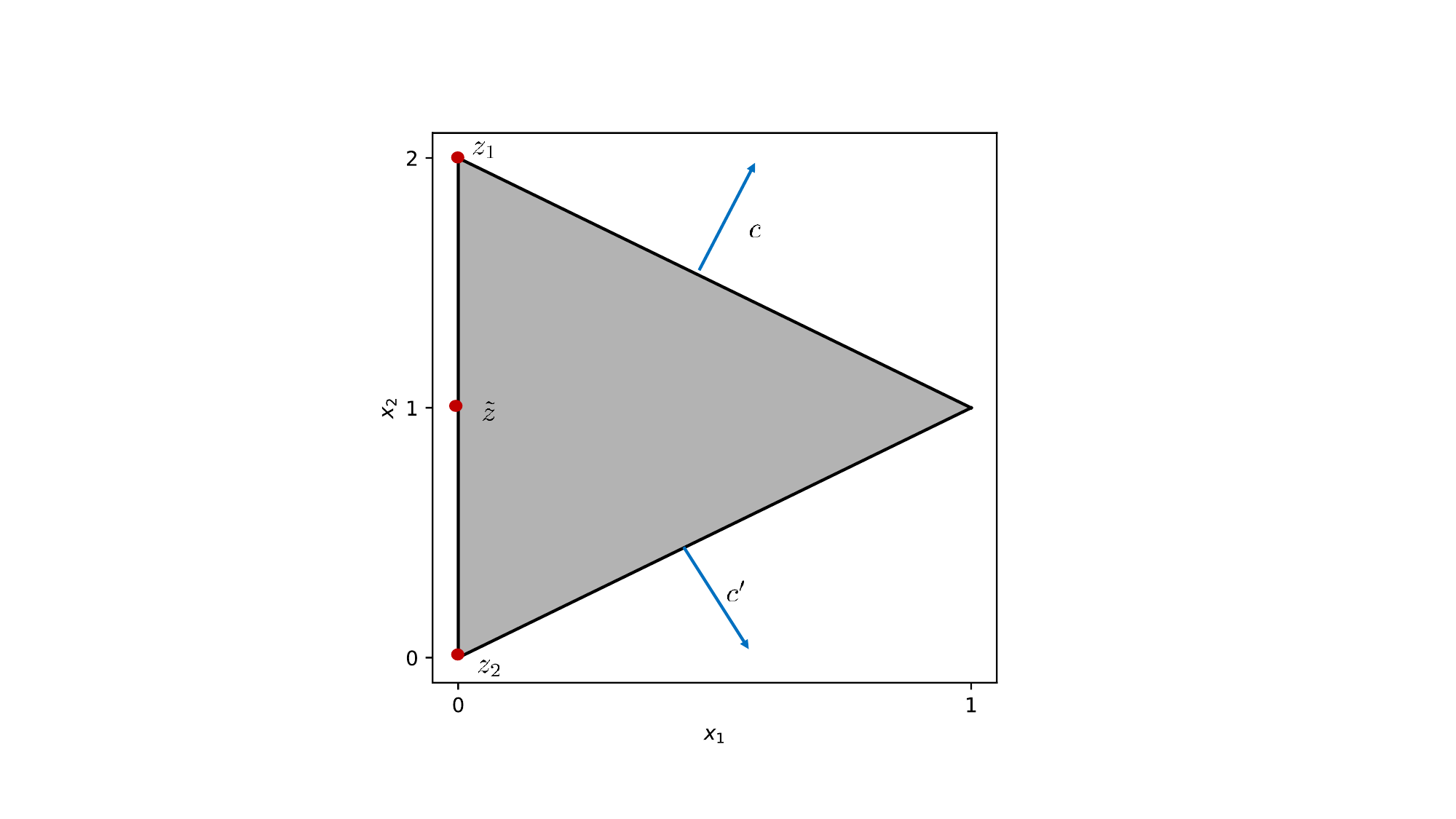}
    \caption{The feasible region of the problem of Example \ref{ex:nonconvex_region_x} (in grey).}
    \label{fig:example_non-convex_on_x}
\end{figure}

The conclusion is that WCEP' is in general a hard non-convex problem. If $D(\hat{x})$, $\mathcal H$, and 
$\delta\left( (c,A,b), (\hat c, \hat A, \hat b)\right)$ are representable by linear constraints, then WCEP' is a bilinear optimization problem.

\subsection{Strong Counterfactual Explanations.}\label{sec:results_strongCE}
The strong counterfactual explanation problem is defined as follows 
\begin{subequations}
\begin{align}
(\text{SCEP}): \quad \inf_{x,c,A,b} \ & \delta\left( (c,A,b), (\hat c, \hat A, \hat b)\right) \label{eq:strong_objective}\\
s.t. \quad & x\in D(\hat x),  \quad \forall \ x\in \argmin_{z: Az\ge b, z\ge 0} c^\top z, \label{eq:strong_optimality}\\
&(c,A,b)\in \mathcal H,\label{eq:strong_mutable_space}
\end{align}
\end{subequations}
where $\delta: \mathcal H \times \mathcal H\to \R_+$ is a given distance function in the parameter space. Note that constraint \eqref{eq:strong_optimality} ensures that all optimal solutions of the problem with selected parameters $(c,A,b)$ are contained in the favored solution space. The SCEP can be interpreted as a pessimistic bilevel problem; see \cite{wiesemann2013pessimistic}.

Note that we are using the infimum instead of the minimum operator in Problem SCEP which is due to the fact that the feasible region in the $(c,A,b)$-parameters can be open as the following example shows.

\begin{example}
Consider the present problem
\begin{align*}
\min \ &-x_1 \\
s.t. \quad & 0\le x_1,x_2\le 1
\end{align*}
and assume that only the objective parameters are mutable. Consider the favored solution space $\mathcal D = \{ x: 0.5\le x_1 \le 1, \ 0\le x_2 \le 0.5\}$. The only face of the feasible region which lies entirely in $\mathcal D$ is the extreme point $(1,0)$. This extreme point is the unique optimal solution for all $(c_1,c_2)\in (0,\infty) \times (-\infty, 0)$ which is an open set. Hence, the feasible region for $(c_1,c_2)$ of the SCEP for this example is open.  
\end{example}

The possibility of an open feasible region for strong CEs is in contrast to the result proved later in Theorem \ref{thm:closedness_relativeCE} which shows that for relative CEs the feasible region is always closed.

In the following theorem we show how to reformulate SCEP into a bilinear problem.

\begin{theorem}\label{thm:SCEP}
Assume the favorable solution space is given by
\[
D(\hat x) = \left\{ x\ge 0: Wx\le h\right\}
\]
for $W\in \R^{q\times n}$ and $h\in \R^q$. Then, SCEP has the same objective value as the following problem:
\begin{align*}
(\text{SCEP'}): \inf_{c,A,b,\Lambda,\Gamma,\tau} \ & \delta\left( (c,A,b), (\hat c, \hat A, \hat b)\right) \\
s.t. \quad & -\Lambda b + \Gamma c \le h, \\
& c \tau^\top  - A^\top \Lambda^\top \ge W^\top, \\
& -b \tau^\top + A\Gamma\ge 0, \\
&(c,A,b)\in \mathcal H, \\
& \Lambda\in \R_+^{q\times m}, \Gamma \in \R_+^{q\times n}, \tau\in \R_+^{q}.
\end{align*}
\end{theorem}
\begin{proof}
First, we can reformulate SCEP as
\begin{align*}
\inf_{c,A,b} \ & \delta\left( (c,A,b), (\hat c, \hat A, \hat b)\right) \\
s.t. \quad & x\in D(\hat x), \quad \forall \ x\in \mathcal U,\\
&(c,A,b)\in \mathcal H,
\end{align*}
where $U=\{ x\ge 0 \ | \  \exists y\ge 0: \ c^\top x\le b^\top y, Ax\ge b, A^\top y\le c\}$. In the description of $U$, we use the same optimality conditions as in the proof of Theorem \ref{thm:weakCE_reformulation}. The latter problem is a robust optimization problem with decision dependent uncertainty set. By using the classical constraint-wise dualization trick from robust optimization, we can reformulate the problem as
\begin{align*}
\inf_{c,A,b} \ & \delta\left( (c,A,b), (\hat c, \hat A, \hat b)\right) \\
s.t. \quad & \max_{x\in \mathcal U} w_i^\top x \le h \quad i=,1\ldots ,p \\
&(c,A,b)\in \mathcal H ,
\end{align*}
where $w_i$ is the $i$-th row of matrix $W$. We can now dualize each of the maximization problems appearing on the left-hand-side separately (introducing a copy of the dual variables for each one) which leads to the formulation presented in the theorem.
\end{proof}

In the following we show that, as for weak CEs, the feasible region of the variables $c,A,b$ may be non-convex and disconnected.

\begin{example}\label{ex:feasible_region_non-convex_strongCE}
Consider again the same present problem in Example \ref{ex:feasible_region_non-convex} together with the same favored solution space. In addition, we assume that only the two objective parameters $c_1,c_2$ can be changed to any value. Then, the set of objective parameters for which every optimal solution is contained in $\mathcal D(\hat x)$ is $\text{int }(\mathcal C_3)\cup  \text{int }(\mathcal C_4)$, where $\mathcal C_3$ and $\mathcal C_4$ are defined as in Example \ref{ex:feasible_region_non-convex}. Note that in contrast to Example \ref{ex:feasible_region_disconnected}, here for the strong CEs we have to choose the interior of the sets since, e.g., for direction $\binom{1}{1}$ the whole line between points $\binom{0}{1}$ and $\binom{1}{0}$ is optimal, and hence, not every optimal solution is contained in $\mathcal D(\hat x)$. Consequently, $\binom{1}{1}$ is not a strong CE. Like in Example \ref{ex:feasible_region_non-convex} and Example \ref{ex:feasible_region_disconnected}, we can show that the feasible region for the strong CEs is non-convex and disconnected.
\end{example}

\subsection{Relative Counterfactual Explanations.}
In the following we consider relative counterfactual explanations. We assume that for the optimal value of the present problem it holds that $\hat c^\top \hat x\ge 0$. Calculating relative counterfactual explanations is computationally easier than weak or strong CEs, since we do not have to handle the optimality condition \eqref{eq:weak_optimality_condition} or \eqref{eq:strong_optimality}. In fact, the relative counterfactual explanation problem can be formulated as 

\begin{subequations}
\begin{align}
(\text{RCEP}):\quad \min_{x,c,A,b} \ & \delta\left( (c,A,b), (\hat c, \hat A, \hat b)\right) \label{constr:bilinear_relative_CE_formulation_a}\\
s.t. \quad & c^\top x \le \alpha\hat c^\top \hat x, \label{constr:bilinear_relative_CE_formulation_b}\\
&Ax\ge b, \label{constr:bilinear_relative_CE_formulation_c}\\
&x\in D(\hat x), \label{constr:bilinear_relative_CE_formulation_d}\\
&(c,A,b)\in \mathcal H \label{constr:bilinear_relative_CE_formulation_e}\\
& x\ge 0. \label{constr:bilinear_relative_CE_formulation_f}
\end{align}
\label{eq:bilinear_relative_CE_formulation} 
\end{subequations}

This is again a bilinear problem. 
While later we will see that the projection on the $x$-space is convex, this may be not the case for the projection on the $(c,A,b)$-space, as the following example illustrates.

\begin{example}\label{ex:RCEP_non-convex_cAb}
Consider the following simple example, with $c=1$, $A=(\alpha, -\alpha, 1,-1)^T$, 
$b=(1,-1,\beta,-\beta)^T$, 
$0\leq x_1 \leq 1$,
$0\leq \alpha \leq 2$, and $0\leq \beta \leq 1$. Then the projection of the feasible region of the relative counterfactual problem ($ \mathcal P$) on the $(\alpha,\beta)$ space is:
\[ \{ (\alpha,\beta) \mid \exists x_1 : x_1  = \beta, \; \alpha x_1 = 1, \; 0\leq x_1 \leq 1, \; 0\leq \alpha\leq 2, \; 0\leq \beta \leq 1 \} = \]
\[\left\{ (\alpha,\beta) \mid  \beta=\frac{1}{\alpha}, \; \frac{1}{2}\leq \alpha \leq 1\right\}, \]
which is not convex.
\end{example}

For the rest of this section we make the following assumptions which we will need to show that the RCEP (in contrast to Example \ref{ex:RCEP_non-convex_cAb}) can be transformed into a convex problem. Furthermore, the assumptions ensure that an optimal solution of the RCEP always exists. In the following we denote $[n]=\left\{ 1,2,\ldots, n\right\}$ for any positive integer $n$.

\begin{assumption}\label{ass:H_columnwise}
The mutable parameter space $\mathcal H$ is compact, convex and \textit{columnwise}, i.e., we have $\mathcal H = \mathcal H_1 \times \ldots \mathcal H_n \times \mathcal H_{n+1}$ where $(c,A,b)\in\mathcal H$ if and only if $(c_j,A_j)\in \mathcal H_j$ for all $j\in [n]$ and $b\in \mathcal H_{n+1}$ and all sets $\mathcal H_1,\ldots , \mathcal H_{n+1}$ are bounded and convex.
\end{assumption} 
\begin{assumption}\label{ass:D_convex}
The favored solution space $\mathcal D (\hat x)$ is a compact and convex set. 
\end{assumption}
\begin{assumption}\label{ass:delta_separable}
The distance measure $\delta$ is continuous and such that
\[
\delta\left( (c,A,b), (\hat c, \hat A, \hat b)\right) = \sum_{j=1}^{n}  \delta_j\left( (c_j,A_j),(\hat c_j,\hat A_j)\right) + \delta_{n+1}\left( b \right), 
\]
where $\delta_j: \R^{m+1}\to \R_+$, for each $j\in [n]$, and $\delta_{n+1}: \R^m \to \R_+$.
\end{assumption}

First we show that under the latter assumptions the feasible region in the $(c,A,b)$-space is compact. The proof is provided in the Appendix. This result is needed to guarantee the existence of an optimal solution of the RCEP.

\begin{theorem}\label{thm:closedness_relativeCE}
    Under Assumption \ref{ass:H_columnwise} -- \ref{ass:delta_separable} the projection of the feasible region of RCEP on the $(c,A,b)$-space is compact.
\end{theorem}

Due to Assumption \ref{ass:delta_separable} the objective function of the RCEP is continuous and hence by the Weierstrass theorem and Theorem \ref{thm:closedness_relativeCE} an optimal solution of the RCEP always exists, leading to the following corollary.

\begin{corollary}
    Under Assumption \ref{ass:H_columnwise} -- \ref{ass:delta_separable} either an optimal solution for the RCEP exists or it is infeasible.
\end{corollary}

We will now discuss how to transform the RCEP into an equivalent convex problem.One special setting which directly leads to a convex optimization problem is the following: Assume the only constraints in the set $D(\hat x)$ are of the form $x_i=\rho_i$, where $\rho_i$ is a fixed value, \textit{i.e.}, the favored solution space requires fixing a subset of variables to given values. If at the same time only the parameters in $c$ and $A$ corresponding to the same columns as the fixed variables are allowed to change, then no bilinear products appear in the formulation RCEP, and hence, the problem is a convex optimization problem if $\delta$ is convex.

To transform RCEP into a convex problem in the more general setting, we use the linear transformation from \cite{gorissen2024}. We will first show that by using the variable transformation
\begin{equation} \label{eq:definition_w_u}
\begin{aligned}
w_j & = c_j x_j\\
u_{ij} & = a_{ij} x_j, \quad \forall i\in[m],
\end{aligned}
\end{equation}
we can reformulate RCEP as a problem with convex feasible region. Then, we show that this variable transformation leads to a convex problem for several relevant classes of distance measures, $\delta$. 

\begin{theorem}\label{thm:RCEP_convex}
The feasible region of RCEP after transformation \eqref{eq:definition_w_u} is convex and can be described as
\begin{subequations}
\begin{align}
    \mathcal P = \{ (x,w,U,b): \   \quad &\bm{1}^\top w \le \alpha \hat c^\top \hat x \label{constr:linearization_relative_CE_a}\\
    & U\bm{1} \ge b \label{constr:linearization_relative_CE_b}\\
    & x\in\mathcal D(\hat x), \label{constr:linearization_relative_CE_c}\\
    & \left(w_j, U_j\right) \in x_j\mathcal H_j, \quad \forall j\in [n], \label{constr:linearization_relative_CE_d}\\
    & b\in\mathcal H_{n+1}, \label{constr:linearization_relative_CE_e}\\
    & x\ge 0 \}, \label{constr:linearization_relative_CE_f}
\end{align}
\label{eq:linearization_relative_CE}
\end{subequations}
where $x_j\mathcal H_j:=\left\{ h': h'=x_jh, \ h\in \mathcal H_j\right\}$ for every $j\in[n]$.
\end{theorem}
\begin{proof}
The proof follows the same idea as in \cite{gorissen2024}. Clearly, constraints \eqref{constr:linearization_relative_CE_a}--\eqref{constr:linearization_relative_CE_c} and \eqref{constr:linearization_relative_CE_f} are equivalent to \eqref{constr:bilinear_relative_CE_formulation_b}--\eqref{constr:bilinear_relative_CE_formulation_d} and \eqref{constr:bilinear_relative_CE_formulation_f} after applying the transformation \eqref{eq:definition_w_u}. Furthermore, if $x_j=0$, then constraint \eqref{constr:linearization_relative_CE_d} enforces $w_j=0$ and $U_j=0$. If $x_j>0$, then Constraint \eqref{constr:linearization_relative_CE_d} is equivalent to
\[
\left( \frac{w_j}{x_j} , \frac{U_j}{x_j}\right) \in \mathcal H_j,
\]
and hence, constraints \eqref{constr:linearization_relative_CE_d} and \eqref{constr:linearization_relative_CE_e}  are equivalent to \eqref{constr:bilinear_relative_CE_formulation_d} by Assumption \ref{ass:H_columnwise}. Note that $x_j\mathcal H_j$ is convex for every fixed $x_j> 0$. Since $\mathcal H_j$ is bounded, it holds $x_j\mathcal H_j= \{ 0\}$ if and only if $x_j=0$, which is again a convex set. Together with Assumption \ref{ass:D_convex} the set $\mathcal P$ is convex which proves the result. 
\end{proof}

From the latter result, we can directly conclude with the following observation.
\begin{corollary}
The projection of the feasible region of Problem \eqref{eq:bilinear_relative_CE_formulation} onto the $x$-variables is convex.
\end{corollary}

The following example shows that if $x$ is not restricted in sign (i.e., $x\geq 0$ is not part of Problem (PP)), then the transformation is invalid, and the projection onto the $x$-space is not convex.

\begin{example}
Consider the following simple example, with $c=1$, $A=(\alpha, -\alpha)^T$, 
$b=(1,-1)^T$, 
$-1\leq x_1 \leq 1$,
$-1\leq \alpha \leq 1$. 
Then, the projection of the feasible region of the relative counterfactual problem ($ \mathcal P$) onto the $x$ space is given by
\[ 
\{ x \mid \exists \alpha:  \; \alpha x = 1, \; -1\leq x \leq 1, \; -1\leq \alpha\leq 1\} \;  = \{-1,1\}, 
\]
which is not convex.
\end{example}

The following example shows that if the columnwise assumption (Assumption \ref{ass:H_columnwise}) does not hold, then the transformation is invalid, and the projection onto the $x$-space is not convex.

\begin{example}
Consider the following simple example, with $c=0$, $A=
\left(\begin{array}{cc}
   \alpha  & -\alpha \\
   -\alpha  & \alpha
\end{array}\right)$, 
$b=(1,-1)^T$, 
$0\leq x_1, x_2 \leq 1$,
$-1\leq \alpha \leq 1$. 
Then, the projection of the feasible region of the relative counterfactual problem ($ \mathcal P$) onto the $x$ space is given by
\[ \{ (x_1,x_2) \mid \exists \alpha:  \; \alpha (x_1-x_2) = 1, \; 
0\leq x_1, x_2 \leq 1, \;
-1\leq \alpha \leq 1\} \;  =  \{(1,0), (0,1)\},
       \]
which is not convex.
\end{example}


Substituting the variable transformation \eqref{eq:definition_w_u} into the objective function of RCEP leads by Assumption \ref{ass:delta_separable} to the objective function
\begin{equation}
\sum_{j=1}^{n}  \delta_j\left( (\frac{w_j}{x_j},\frac{U_j}{x_j}),(\hat c_j,\hat A_j)\right) + \delta_{n+1}\left( b \right).
\end{equation}
There are several relevant cases in which RCEP remains convex after applying the variable transformation \eqref{eq:definition_w_u}:

\begin{enumerate}
  \item Suppose only the parameters in $b$ and one column $j \in [n]$ are allowed to change. Hence, in this case the objective function becomes 
  \[\delta_j((c_j,A_j), (\hat{c}_j,\hat{A}_j)) + \delta_{n+1}(b, \hat{b}),\]
  in which $\delta_j$ and $\delta_{n+1}$ are convex functions. Using an epigraph variable $z$, the objective is to minimize $z+ \delta_{n+1}(b, \hat{b})$, under the additional constraint 
  \[ \delta_j((c_j,A_j),(\hat{c}_j,\hat{A}_j))\leq z.\] 
  If we fix the value $z$, then we can add this constraint to the set of convex constraints that define $\mathcal{H}_j$. Hence, in this case, RCEP reduces to a convex optimization problem. We can solve RCEP by using binary search over the possible values of $z$. 
  \item Suppose the objection function is given by
  \[  \max_{j\in [n]} \delta_j((c_j,A_j),(\hat{c}_j,\hat{A}_j)) + \delta_{n+1}(b, \hat{b}). \]
  Again, we use an epigraph variable $z$ for this objective function. Hence, the objective changes into minimizing $z+\delta_{n+1}(b, \hat{b})$, and one extra constraint \[\delta_j((c_j,A_j),(\hat{c}_j,\hat{A}_j))\leq z\]
  is added to $\mathcal{H}_j$ for every $j\in [n]$ . As in the previous case, this yields a convex optimization problem for a fixed value of $z$, and problem RCEP can be solved by using binary search over the value of $z$. 
  \item Suppose the objective function is given by
  \begin{equation}\label{eq:xdelta_obj}  \sum_{j=1}^n x_j\delta_j((c_j,A_j),(\hat{c}_j,\hat{A}_j))  + \delta_{n+1}(b, \hat{b}). \end{equation}
  One motivation for this choice is that  changes in the parameters of a certain column are counted less in the penalty function when the corresponding $x_j$ is low. Another situation could be when only the parameters in one column are mutable, and in $D(\hat{x})$ the value of the corresponding $x_j$ is specified.  
  After substitution \eqref{eq:definition_w_u}, this objective function becomes
  \[  \sum_{j=1}^n x_j\delta_j\left(\left(\frac{w_j}{x_j},\frac{U_j}{x_j}\right),(\hat{c}_j,\hat{A}_j)\right) + \delta_{n+1}(b, \hat{b}), \]
  which is jointly convex in $(x,w,U,b)$. Moreover, if $\delta_j((c_j,A_j),(\hat{c}_j,\hat{A}_j))$, $j \in [n]$, are linearly representable (e.g. the $\ell_1$- or $\ell_\infty$-norm) then, RCEP is linearly representable.
  \item Suppose the objective function is given by
  \[  \sum_{j=1}^n \delta_j((c_jx_j,A_jx_j),(\hat{c}_j\hat{x}_j,\hat{A}_j\hat{x}_j)) + \delta_{n+1}(b, \hat{b}). \]
After substituting \eqref{eq:definition_w_u}, this objective function becomes
  \[  \sum_{j=1}^n \delta_j((w_j,U_j),(\hat{c}_j\hat{x}_j,\hat{A}_j\hat{x}_j)) + \delta_{n+1}(b, \hat{b}), \]
which is jointly convex in $(w,U,b)$.
\end{enumerate}

\subsection{Testing Feasibility.}
\label{sec:tests}
In the following, we provide efficient ways to test if a returned parameter setting $(c,A,b)$ is actually feasible for the three different types of problems:

\begin{itemize}
    \item To verify that a solution $(c,A,b)$ is a relative CE, one can solve
    \begin{align*}
        \min_x \ & c^\top x \\
        s.t. \quad & Ax\ge b, \\
        &  x\in \mathcal D(\hat x), \\
        & x\ge 0.
    \end{align*}
    Then, we can check if the optimal value is at most $\alpha \hat c^\top \hat x$. If this is the case, then the solution $(c,A,b)$ is a relative CE.
    \item To verify that a solution $(c,A,b)$ is a weak CE, the equality in Proposition \eqref{prop:weak_CEs} can be checked by solving both the standard linear problem and the same problem with the additional constraints in $\mathcal D(\hat x)$. If both problems have the same optimal objective function value, then $(c,A,b)$ is a weak CE.
    \item There is an efficient preliminary-test that can detect whether solution $(c,A,b)$ is not a strong CE. This can be done by simply solving  
    \begin{align*}
        \min_x \ & c^\top x \\
        s.t. \quad & Ax\ge b, \\
        & x\ge 0,
    \end{align*}
    and checking if the optimal solution fulfills all the requirements in $\mathcal D(\hat x)$. If this is not the case, then the solution cannot be a strong CE. To verify that the solution is a strong CE, we have to check whether the following problem is feasible
    \begin{align*}
        \min_x \ & c^\top x \\
        s.t. \quad & x \in \mathcal D(\hat x), \quad \forall x \in \argmin_{\substack{z: Az\ge b, z\ge 0}} \ c^\top z.
    \end{align*}
    The robust constraint can again be reformulated into a finite set of constraints by following the steps in the proof of Theorem \ref{thm:SCEP}.
\end{itemize}

\section{Numerical Experiments.}
In this section,  we present three experiments. First, we calculate all three types of CEs (relative, weak and strong) for a low dimensional example of the diet problem and present the corresponding solutions to give a better understanding of the concepts. Second, we run experiments for relative and weak CEs for the complete version of the diet problem, which was studied in \cite{maragno2021mixed}, and present the corresponding solutions. The results indicate that calculating optimal weak or strong CEs 
is not always possible during the time limit of one hour. On the other hand, optimal relative CEs can be calculated in seconds. Based on the latter results, we propose to use the concept of relative CEs for large dimensional instances. Hence, in the third experiment we calculate relative CEs for all NETLIB instances \citep{netlib}. We compare the bilinear formulation \eqref{eq:bilinear_relative_CE_formulation} to the equivalent linear reformulation \eqref{eq:linearization_relative_CE} with the objective function \eqref{eq:xdelta_obj}. The experiments show that the latter can be solved significantly faster.

All algorithms were implemented in Python 3 and executed on a cluster with CPU AMD Rome 7H12 (2x), 64 Cores/Socket, 2.6GHz and 16 x 16GiB
3200MHz DDR4 RAM. All optimization problems are solved by Gurobi 10.0.2. 

We found stark numerical instabilities in our experiments when solving the bilinear formulations for weak, and strong CEs. These instabilities are probably caused by the possibility that the feasible region of the counterfactual problems can be open as shown in Sections \ref{sec:results_weakCE} and \ref{sec:results_strongCE}. Changing the accuracy parameters of Gurobi, or switching off the presolve procedure, could lead to significantly different solutions and sometimes it could not be verified (by applying the tests in Section \ref{sec:tests}) that the solution is indeed a weak or strong CE.

\subsection{The Diet Problem.}
In this section, we illustrate the different types of CEs in a practical setting. To this end, we consider the following version of the diet problem, which was studied in \cite{maragno2021mixed}.

We consider a set of suppliers $\mathcal S$ and a set of different types of food $\mathcal F$, containing, \textit{e.g.}, wheat, sugar or beans. Each food type contains a set of nutrients $\mathcal N$. The purchasing price per $100$ g of food type $f$ at supplier $s$ is denoted as $p_{f}^s$. We denote the vector of the prices  for all food types of supplier $s$ as $p^s\in\R_+^{|\mathcal F|}$. For each food type $f$, the amount of nutrient $\nu\in \mathcal N$ per 100g is denoted as $w_f^\nu$. For each nutrient, the required amount in grams per day is given as $\text{req}_{\nu}$. The goal is to decide how much of each food to buy from which supplier, such that the whole food basket 
satisfies all nutrient requirements and has minimal costs.

The diet problem is then defined as
\begin{equation}\label{eq:dietproblem}
\begin{aligned}
    \min \ &\sum_{s\in \mathcal S} (p^s)^\top x^s \\
    s.t. \quad & \sum_{s\in \mathcal S} (w^\nu)^\top x^s \ge \text{req}_\nu, \forall \nu\in \mathcal N, \\
    & \sum_{s\in \mathcal S} x_{\text{sugar}}^s = 0.2, \\
    & \sum_{s\in \mathcal S} x_{\text{salt}}^s = 0.05, \\
    & x^s\ge 0 \ \forall s\in \mathcal S,
\end{aligned}
\end{equation}
where 
$x^s_{\text{product}}$ denotes the amount of the product (in 100g) which is purchased from supplier $s$. To reduce numerical instability we set an upper bound of $100$ for each variable, \textit{i.e.}, $0\le x_j^s\le 100$ for all $j\in \mathcal F$ and $s\in\mathcal S$. We use the Syria dataset as used in \cite{maragno2021mixed}. In the following subsections, we choose the $\ell_1$-norm for the distance measure $\delta_j$.

\subsubsection{Reduced Problem.}\label{sec:reduced_diet_problem}
To explain the concepts of the different types of counterfactual explanations, we consider a reduced version of the diet problem with only two suppliers, three products, and three nutrients. We round most parameters to the closest integer value:
\begin{equation}\label{eq:toy_dietproblem}
\begin{aligned}
    \min \ & 800 x_{\text{beans}}^1 + 1003 x_{\text{rice}}^1 + 300 x_{\text{wheat}}^1 + 1434 x_{\text{beans}}^2 + &1336 x_{\text{rice}}^2 + 500 x_{\text{wheat}}^2 \\
    s.t. \quad &  \sum_{s=1}^{2} 335 x_{\text{beans}}^s + 360 x_{\text{rice}}^s+ 330 x_{\text{wheat}}^s
   \ge 2100,  &\text{(Energy(kcal))}\\
   &  \sum_{s=1}^{2} 20 x_{\text{beans}}^s + 7 x_{\text{rice}}^s+ 12 x_{\text{wheat}}^s
   \ge 52.5,  &\text{(Protein(g))}\\
   &  \sum_{s=1}^{2}  x_{\text{beans}}^s +  \frac{1}{2}x_{\text{rice}}^s+ 2 x_{\text{wheat}}^s
   \ge 35, &\text{(Fat(g))}\\
    & 0\le x^s\le 100, \quad \forall s=1,2.
\end{aligned}
\end{equation}
The optimal solution of the latter problem is to buy $1750$ gram of wheat from Supplier 1, leading to the optimal costs of $5250$ (denoted as $x_{\text{wheat}}^1=1750$). Since Supplier 2 does not sell anything when this solution is implemented by the decision maker, she could ask the following counterfactual question: ``What is the minimum change in my prices such that at least 100 grams of beans and 250 grams of rice would be purchased from me?''

Formally the favored solution space is given as $\mathcal D=\left\{ x: x_{\text{beans}}^2\ge 1, \ x_{\text{rice}}^2\ge 2.5\right\}$. We allow every price of Supplier 2 to be changed by up to $\pm 100\%$, which defines the mutable parameter space. The factor $\alpha$ for relative CEs is set to one. The optimal CEs for the latter question can be found in Table \ref{tbl:CEs_prices_ToyDiet}.

\begin{table}[h!]
    \centering
    \begin{tabular}{l|c|c|c}
        & relative CE & weak CE & strong CE \\
        \hline
        CE & $p_{\text{wheat}}: 500.0 \to 29.1$ & \makecell{$p_{\text{beans}}: 1434.0 \to 150.0$ \\ $p_{\text{rice}}: 1336.0 \to 75.0$} & \makecell{no feasible solution \\ found during TL} \\
        \hline
        Opt. sol. & $x_{\text{wheat}}^2=1750.0$g &  \makecell{$x_{\text{wheat}}^1=1750$g} & - \\
        \hline
        \makecell{Opt. sol. with $D(x)$} & \makecell{$x_{\text{beans}}^2 = 100.0$g \\ $x_{\text{rice}}^2 = 250.0$g \\ $x_{\text{wheat}}^2 = 1637.5$g}& \makecell{$x_{\text{beans}}^2=100.0$g \\ $x_{\text{rice}}^2=6800$g}&  - 
    \end{tabular}
    \caption{Optimal counterfactual explanations regarding the prices for Problem \eqref{eq:toy_dietproblem} (Row 1); returned optimal solution of Problem \eqref{eq:toy_dietproblem} after price changes have been applied (Row 2); returned optimal solution of Problem \eqref{eq:toy_dietproblem} after price changes have been applied and with additional constraints in $\mathcal D$ (Row 3).}
    \label{tbl:CEs_prices_ToyDiet}
\end{table}

For the relative CE only the price of wheat is reduced significantly which leads to an optimal solution where only wheat is bought from Supplier 2 instead of the desired amounts of beans and rice. This can happen since a relative CE only asks for the minimal change in the prices such that the optimal solution, when adding the constraints in $\mathcal D$ to the problem, has at most the costs of the original problem. However, if the decision maker does not add these constraints to the problem then her optimal solution can violate the constraints in $\mathcal D$ and even have a smaller objective value than the one determined when the constraints are added.

For the weak CEs, two price parameters have to be changed, namely the prices for beans and for rice. However, when the problem is optimized with the updated prices the optimal solution does not fulfill the requirements in $\mathcal D$. This can happen since multiple optimal solutions exist and not the desired one is chosen by the solution method. However, if we optimize the same problem with additional constraints in $\mathcal D$, then the optimal value does not change but the solution fulfills the constraints in $\mathcal D$. For the strong CEs, no feasible CE could be found during the time limit of $2$ hours.

If Supplier 2 is not satisfied with the price changes which she has to undertake to achieve her desired solution properties, she could consider changing the nutrient values of her products as well. A counterfactual question in this case would be: ``What is the minimum change in my prices and the nutrient values of my products which I have to undertake, such that at least 100 grams of beans and 250 grams of rice would be purchased from me?'' The optimal CEs for this counterfactual question can be found in Table \ref{tbl:CEs_Nutrients_ToyDiet} in the Appendix.

\begin{table}[h!]
    \centering
    \resizebox{\columnwidth}{!}{
    \begin{tabular}{c|c|c|c}
        & relative CE & weak CE & strong CE \\
        \hline
        CE & \makecell{$p_{\text{wheat}}: 500.0 \to 62.4$ \\ Fat(g) in Beans : $1.0 \to 2.0$ \\ Fat(g) in Rice : $0.5 \to 1.0$ \\ Fat(g) in Wheat : $2.0 \to 4.0$} &  \makecell{$p_{\text{beans}}: 1434.0 \to 1000.0$ \\ $p_{\text{rice}}: 1336.0 \to 8.8$ \\ Protein(g) in Beans: $20.0 \to 40.0$ \\ Protein(g) in Rice: $7.0 \to 0.35$ \\ Fat(g) in Rice: $0.5 \to 1.0$} & \makecell{$p_{\text{beans}}: 1434.0 \to 0.0$ \\ $p_{\text{rice}}: 1336.0 \to 327.3$ \\ Energy(kcal) in Beans : $335.0 \to 0.0$ \\ Fat(g) in Rice : $0.5 \to 0.0$} \\
        \hline
        Opt. sol. & $x_{\text{wheat}}^2=875$g & \makecell{$x_{\text{beans}}^2=100.0$g \\ $x_{\text{rice}}^2=3400.0$g} & \makecell{$x_{\text{beans}}^2=100.0$g \\ $x_{\text{rice}}^2=583.3$g}\\
        \hline
        \makecell{Opt. sol. \\ with $\mathcal D$} & \makecell{$x_{\text{beans}}^2=100$g \\ $x_{\text{rice}}^2=250.0$g \\ $x_{\text{rice}}^2=762.5$g}  & \makecell{$x_{\text{beans}}^2=100.0$g \\ $x_{\text{rice}}^2=3400.0$g}  & \makecell{$x_{\text{beans}}^2=100.0$g \\ $x_{\text{rice}}^2=583.3$g}
    \end{tabular}}
    \caption{Optimal counterfactual explanations regarding the prices and nutrient values for Problem \eqref{eq:toy_dietproblem} (Row 1); returned optimal solution of Problem \eqref{eq:toy_dietproblem} after price and nutrient changes have been applied (Row 2); returned optimal solution of Problem \eqref{eq:toy_dietproblem} after price and nutrient changes have been applied and with additional constraints in $\mathcal D$ (Row 3).}
    \label{tbl:CEs_Nutrients_ToyDiet}
\end{table}

The results show that for the relative CEs, only the price for wheat is reduced significantly while the amount of fat for all products is doubled. Note that we only allow a change of $100\%$ for each parameter, hence more than doubling the nutrient value is not possible. For the weak CEs both, the prices of beans and rice have to be reduced, while at the same time three nutrient values have to be changed. This time it can be seen that the optimal solution after applying these changes fulfills the requirements in $\mathcal D$, although we only solve the weak CEs. For the strong CEs less parameters are changed compared to the weak CEs. However, the $\ell_1$-norm of parameter changes is larger than the one for weak CEs which provably must be the case. It can be seen that here optimizing the updated problem without the constraints in $\mathcal D$ leads to the desired requirements in $\mathcal D$ which must be the case for strong CEs. Interestingly, finding a strong CE was possible during the time limit, while for the setup, where only objective parameters can be changed, this was not possible; see Table \ref{tbl:CEs_prices_ToyDiet}.

\subsubsection{Complete Problem.}
We now study the diet problem in its full dimension. An optimal solution of the present problem \eqref{eq:dietproblem} calculated by Gurobi is given in Table \ref{tbl:optimal_solution_diet_problem}.

\begin{table}[h!]
    \centering
    \begin{tabular}{c|c|c|c}
        Supplier ID & Prod. 1 (amount) & Prod. 2 (amount) & Prod. 3 (amount)  \\
        \hline
        $1$ & Wheat (263.3 g) & & \\
        $2$ & CSB (70 g) & & \\
        $3$ & & & \\
        $4$ & Salt (5g) &  Oil (21.86 g) & \\
        $5$ & Milk (69.5 g) & Maize meal (122.6 g) & Sugar (20g)
    \end{tabular}
    \caption{Optimal solution of the diet problem.}
    \label{tbl:optimal_solution_diet_problem}
\end{table}

In the following we calculate strong, weak and relative CEs for the latter problem, where first we study counterfactuals only in the price parameters (\textit{i.e.}, objective parameters) and afterwards we consider counterfactuals in both, price parameters and nutrient parameters (\textit{i.e.}, objective and constraint parameters).

We implemented formulation \eqref{eq:bilinear_relative_CE_formulation} to calculate the relative CEs, the formulation from Theorem \ref{thm:weakCE_reformulation} to calculate weak CEs, and the formulation from Theorem \ref{thm:SCEP} to calculate strong CEs. All formulations were implemented in Gurobi with a time limit of one hour. All calculations for the relative CEs were finished in milliseconds while the calculations for the weak CEs sometimes took up to one minute. For the strong CEs, the numerical instabilities were very stark often leading to solutions which could not be verified to be strong CEs. Furthermore, when increasing the accuracy parameters often no feasible strong CEs could be found during the time limit. Hence, we do not present results for the strong CEs in this section. For the weak CEs we set the Gurobi parameters \textit{FeasibilityTol} and \textit{OptimalityTol} to $10^{-9}$ and switched off the presolve procedure. All solutions in the following tables are the best known solutions found during the time limit.  

\begin{table}[h!]
    \centering
    \begin{tabular}{c|c|c|c|c}
        \textbf{Supplier} & \multicolumn{2}{c|}{\textbf{relative CEs}} & \multicolumn{2}{c}{\textbf{weak CEs}}  \\
        \hline
        &CE&t (in s)&CE&t (in s) \\
        \hline
        \hline
        $1$ & \makecell{$p_{\text{wheat}}: 300.0 \to 112.4$} & $0.2$ &\makecell{$p_{\text{beans}}: 800.0 \to 255.4$ \\ $p_{\text{wheat}}: 300.0 \to 227.6$} & $60.3$  \\
        \hline
        $2$ & \makecell{$p_{\text{beans}}: 1433.5 \to 1407.8$ \\ $p_{\text{csb}}: 799.0 \to 792.5$ \\ $p_{\text{wheat}}: \infty \to 0.0$} & $0.1$ & \makecell{$p_{\text{beans}}: 1433.5 \to 233.6$ \\ $p_{\text{csb}}: 799.0 \to 839.0$ \\$p_{\text{wheat}}: \infty \to 233.2$} & $10.9$  \\
        \hline
        $3$ & \makecell{$p_{\text{wheat}}: 300.0 \to 112.4$} & $0.1$ & \makecell{$p_{\text{beans}}: 800.0 \to 255.4$ \\ $p_{\text{wheat}}: 300.0 \to 227.6$} & $60.4$  \\
        \hline     
        $4$ & \makecell{$p_{\text{wheat}}: \infty \to 13.5$} & $0.1$ & \makecell{$p_{\text{beans}}: 1151.1 \to 223.1$ \\ $p_{\text{wheat}}: \infty \to 221.2$} & $4.8$  \\
        \hline
        $5$ & \makecell{$p_{\text{wheat}}: \infty \to 1.0$} & $0.1$ & \makecell{$p_{\text{beans}}: 1383.2 \to 295.5$ \\ $p_{\text{bulgur}}: 671.4 \to 676.3$ \\ $p_{\text{milk}}: 467.5 \to 504.2$ \\ $p_{\text{maize meal}}: 249.3 \to 289.3$ \\ $p_{\text{wheat}}: \infty \to 266.2$} & $17.1$ \\
        \hline
    \end{tabular}
    \caption{Best known counterfactual explanations regarding changes in prices after time limit of one hour.}
    \label{tbl:CEs_prices_Diet}
\end{table}

\paragraph*{Prices.}\label{sec:price_CE_dietproblem}
We assume now that each supplier $s$ has control over its own prices $p^s$ which can be changed by at most $\pm 100\%$. Each supplier asks the following counterfactual question: ``What is the minimum change in my prices which I have to undertake, such that at least 100 grams of beans and 250 grams of wheat would be purchased from me?''

The solutions for this question are shown in Table \ref{tbl:CEs_prices_Diet}. The results show that the required price changes for the relative CEs are smaller than the ones for the weak CEs. However, some suppliers do not provide wheat, hence the change of price for wheat is infinite. In the Syria data, the value for $\infty$ is set to $1,000,000$. While the relative CEs could be solved in milliseconds the solution times for weak CEs could take up to one minute. 

\paragraph*{Nutrient Values.}\label{sec:nutrient_CE_dietproblem}
We now assume that each supplier can additionally change the amount of nutrients of each of her products by at most $\pm 100\%$. Note that the nutrient values of the other suppliers are not changed. Then, the question is ``What is the minimum change in the prices and nutrient values of all my products which I have to undertake, such that at least 100 grams of beans and 250 grams of wheat would be purchased from me?''

The best known solutions after the time limit are presented in Table \ref{tbl:CEs_Nutrients_Diet}. The results show that the number of changed parameters and the magnitude of changes is much smaller for relative CEs. For weak CEs the number of nutrient parameters to change is large. This can also be due to the restriction that each parameter can be changed by at most $\pm 100\%$. Indeed, it can be seen that many nutrient parameters were set to the maximum feasible value. All computations for the relative CEs could be performed in milliseconds, while all computations for the weak CEs hit the time limit.

\subsection{NETLIB Instances.}
In this section, we perform experiments for relative CEs on the NETLIB instances \citep{netlib}. To this end we calculate CEs by formulation \eqref{eq:bilinear_relative_CE_formulation} (where $\delta_j$ is the $\ell_1$-norm) and by the convex reformulation \eqref{eq:linearization_relative_CE} with objective \eqref{eq:xdelta_obj}, where again $\delta_j$ is the $\ell_1$-norm. Using the resulting objective function, we finally obtain a linear optimization even if we have 
an arbitrary number of columns which may contain mutable parameters. This significantly speeds up the solution time.

We next describe the setup for the experiments. For every NETLIB instance, we calculate the optimal solution $\hat x$. We generate the favored solution space $\mathcal D(\hat x)$ as follows: for every instance, we select three random columns $j$. If there is no conflict with the variable bounds of the NETLIB instance the favored constraint $x_j\ge 1.05\hat x_j$ is added (otherwise $x_j\le 0.95\hat x_j$). If $\hat x_j=0$ the favored constraint $x_j\ge 0.05$ is added if there is no conflict with the original variable bounds (otherwise $x_j\le -0.05$). To generate the mutable parameter space $\mathcal H$ we do the following. We draw one, five, and $10$ random columns, where we only consider columns for which the original lower bound on the variable in the present problem is non-negative. In each of the selected columns, all parameters are defined as mutable when rounding to zero digits changes the value. For a small set of instances, all problem parameters are integer, hence the latter procedure does not provide any mutable parameters. For this set of instances, all parameters are considered as mutable whose absolute value is larger than 10 and is not a multiple of 10. If we draw five columns we combine these with the column from the iteration where we only draw one column. When we draw $10$ columns we combine these with the columns from the iterations where we draw one column and five columns. By this procedure we ensure that the optimal value for a larger number of columns is at least as good as for a smaller number, since all columns from the previous iterations are contained. Every mutable parameter is allowed to change by $\pm 100\%$.  

The optimality factor is set to $\alpha=1$, i.e., we are looking for the smallest change in the problem parameters such that a solution with the desired properties in $\mathcal D(\hat x)$ exists, and which has the objective function value at least as good as the present problem. For every NETLIB instance and for every number of columns in $\{1,5,10\}$, we generate $20$ CE instances as described above. We divide the NETLIB instances into nine categories where we classify the dimension $n$ of the problem (\textit{i.e.}, the number of variables) as small if $n$ is at most the $35\%$ quantile, as medium if $n$ is between the $35\%$ and the $70\%$ quantile, and as large if $n$ is above the $70\%$ quantile of the list of dimensions of all NETLIB instances. We classify the number of constraints $m$ of each problem by the same procedure. The detailed information for the categorization is shown in Table \ref{tbl:NETLIB_categorization}.

\begin{table}[h!]
\centering
\resizebox{0.5\columnwidth}{!}{
\begin{tabular}{ccc}
Type & Category & Intervals \\
\hline
\# Variables & small & $0\le n\le 534$ \\
\# Variables & medium & $534\le n\le 2167$ \\
\# Variables & large & $2167\le n\le 22275$ \\
\hline
\# Constraints & small & $0\le m\le 351$ \\
\# Constraints & medium & $351\le m\le 906$ \\
\# Constraints & large & $906\le m\le 16675$ \\
\end{tabular}
}
\caption{Categorization of NETLIB instances.}\label{tbl:NETLIB_categorization}
\end{table}

Table \ref{tbl:NETLIB_instanceInfo} contains information on the instance setup. Note that the two columns denoted by ``\# mutable objective param.'' and ``\# mutable constraint param.''  correspond to the average number of bilinear terms which appear in formulation \eqref{eq:bilinear_relative_CE_formulation}. This value can be different for each NETLIB instance due to the procedure how we select mutable parameters as described above.

\begin{table}[h!]
\centering
\resizebox{0.7\columnwidth}{!}{
\begin{tabular}{l|l|l|rrrr}
$n$ & $m$ & \makecell{\# mut.\\columns} & \# inst. & \makecell{feasible \\ (in \%)} & \makecell{\# mutable \\ objective \\ param.} & \makecell{\# mutable \\ constraint \\ param.} \\
\hline
\multirow[c]{6}{*}{small} & \multirow[c]{3}{*}{small} & 1 & 28 & 38.00 & 0.38 & 4.85 \\
 &  & 5 & 28 & 54.00 & 1.67 & 20.93 \\
 &  & 10 & 28 & 59.00 & 4.09 & 54.67 \\
 & \multirow[c]{3}{*}{medium} & 1 & 7 & 36.00 & 0.86 & 6.09 \\
 &  & 5 & 7 & 61.00 & 4.01 & 32.48 \\
 &  & 10 & 7 & 64.00 & 10.78 & 81.00 \\
 \hline
\multirow[c]{9}{*}{medium} & \multirow[c]{3}{*}{small} & 1 & 4 & 84.00 & 0.75 & 5.88 \\
 &  & 5 & 4 & 100.00 & 2.59 & 22.58 \\
 &  & 10 & 4 & 100.00 & 6.75 & 54.12 \\
 & \multirow[c]{3}{*}{medium} & 1 & 22 & 35.00 & 0.47 & 20.00 \\
 &  & 5 & 22 & 51.00 & 2.41 & 26.66 \\
 &  & 10 & 22 & 58.00 & 6.35 & 42.81 \\
 & \multirow[c]{3}{*}{large} & 1 & 8 & 31.00 & 0.56 & 3.51 \\
 &  & 5 & 8 & 55.00 & 2.07 & 11.92 \\
 &  & 10 & 8 & 63.00 & 5.66 & 27.24 \\
 \hline
\multirow[c]{9}{*}{large} & \multirow[c]{3}{*}{small} & 1 & 2 & 25.00 & 1.00 & 0.00 \\
 &  & 5 & 2 & 48.00 & 1.89 & 0.42 \\
 &  & 10 & 2 & 50.00 & 4.05 & 1.15 \\
 & \multirow[c]{3}{*}{medium} & 1 & 6 & 43.00 & 0.63 & 2.49 \\
 &  & 5 & 6 & 49.00 & 3.34 & 10.46 \\
 &  & 10 & 6 & 53.00 & 8.62 & 26.16 \\
 & \multirow[c]{3}{*}{large} & 1 & 21 & 45.00 & 0.60 & 5.49 \\
 &  & 5 & 21 & 58.00 & 2.50 & 22.75 \\
 &  & 10 & 21 & 65.00 & 6.40 & 54.14 \\
\end{tabular}
}
\caption{Instance informations for the considered classes of NETLIB instances. We show from left to right the following values: the size of the dimension; the size of the number of constraints; the number of random columns for mutable parameter selection; the number of NETLIB instances which fall into the category; the number of instances for which a feasible relative CE exists in \%; the average number of objective parameters which are selected to be mutable; the average number of constraint parameters which are selected to be mutable. }\label{tbl:NETLIB_instanceInfo}
\end{table}

Table \ref{tbl:NETLIB_solutionInfo} contains information about the solution quality of both problems. In the solutions of the linear problem \eqref{eq:linearization_relative_CE}, the number of changed parameters (both objective and constraint parameters) is smaller than for problem formulation \eqref{eq:bilinear_relative_CE_formulation}, \textit{i.e.}, the linear problem leads to sparser changes in the problem parameters. This is probably due to the convexity of the linear  problem, since minimizing the $\ell_1$-norm over convex sets leads to sparse solutions.

Table \ref{tbl:NETLIB_solutionTimeInfo} shows results on the runtime of the methods. The results show that for all instance sizes the linear problems can be solved faster than the bilinear formulation. The improvement in solution time increases for larger instance sizes. The same holds for detecting infeasibility. Furthermore, the bilinear formulation could not be solved to optimality in all cases within the time limit, while the linear problems never hit the time limit. We also show the time to solve the present problem (PP) and the time to set up the problem formulation for the present problem. The latter metric is motivated by our observation that setting up the constraint matrix (although using the sparse-matrix datatype provided by Numpy) takes a significant amount of the runtime for both, the present problem and the counterfactual problems. The results show that the runtime of the present problem is of the same order or sometimes even larger order compared to the runtime of solving the linear relative CE problem. The larger runtime comes from the fact that not all of the relative CE problems are feasible.

\section{Outlook.}
In this work, we argue that counterfactual explanations constitute a useful tool to provide explainability for linear optimization problems. We present three different types of counterfactual explanations which cover many relevant situations in practical applications. In contrast to weak CEs, the concept of strong CEs considers the case that more than one optimal solution exists and enforces that all of them fulfill the desired conditions. On the other hand relative CEs provide changes in the problem paramaters for which a desired solution would not lead to a large increase in objective value. Our theoretical analysis shows that the relative CE problem can be reformulated as a convex problem for many special cases. Using this hidden convexity leads to computationally tractable solution methods as our experiments confirm.

Since the development of counterfactual explanations in optimization is still at the very beginning, there are many open research questions which should be studied in future works. The concepts we developed in this work could be extended to non-linear optimization problems, or to problems involving integer decisions. While for the latter class of problems several works already exist, these do not cover the case of mutable constraint parameters.

While we developed the concept of strong CEs to tackle the issue of possible multiple optimal solutions, our experiments show that the changes in problem parameters which have to be performed for a strong CE are significant or in many situations strong CEs even do not exist. This is undesirable in a practical setting. The reason for developing this concept stems from the fact that we do not consider which solution algorithm is used by the decision maker to derive an optimal solution. A future direction could be to calculate CEs for specific types of optimization algorithms, \textit{e.g.}, the simplex method or the branch-and-bound method. While this is a challenging task, this would provide a less conservative way to calculate CEs in the strong fashion.

\begin{table}[h!]
\centering
\resizebox{0.7\columnwidth}{!}{
\begin{tabular}{l|l|l|rr|rr}
\multicolumn{3}{c|}{} & \multicolumn{2}{c|}{\makecell{\eqref{eq:bilinear_relative_CE_formulation} with obj. \\$\sum_{j=1}^{n} \delta_j((c,A),(\hat c,\hat A))$}} & \multicolumn{2}{c}{\makecell{\eqref{eq:linearization_relative_CE} with obj. \\$\sum_{j=1}^{n} x_j\delta_j((c,A),(\hat c,\hat A))$}} \\
\hline
$n$ & $m$ & \makecell{\# mut.\\col.} & $\|c-\hat c\|_0$ & $\| A-\hat A\|_0$  & $\|c-\hat c\|_0$ & $\| A-\hat A\|_0$ \\
\hline
\multirow[c]{6}{*}{small} & \multirow[c]{3}{*}{small} & 1 & 0.25 & 2.24  & 0.18 & 1.14  \\
 &  & 5 & 0.46 & 4.39& 0.20 & 1.20  \\
 &  & 10 & 0.42 & 6.08  & 0.17 & 1.34  \\
 & \multirow[c]{3}{*}{medium} & 1 & 0.63 & 2.30  & 0.55 & 1.04 \\
 &  & 5 & 0.84 & 4.56  & 0.45 & 2.03  \\
 &  & 10 & 0.76 & 5.31  & 0.34 & 2.88  \\
 \hline
\multirow[c]{9}{*}{medium} & \multirow[c]{3}{*}{small} & 1 & 0.49 & 2.13  & 0.38 & 1.66  \\
 &  & 5 & 0.50 & 4.08 & 0.09 & 2.17  \\
 &  & 10 & 1.06 & 10.92 & 0.00 & 1.99  \\
 & \multirow[c]{3}{*}{medium} & 1 & 0.33 & 2.59  & 0.31 & 0.92  \\
 &  & 5 & 0.42 & 3.23  & 0.31 & 0.98  \\
 &  & 10 & 0.58 & 4.31  & 0.32 & 0.96  \\
 & \multirow[c]{3}{*}{large} & 1 & 0.49 & 1.37  & 0.47 & 0.75 \\
 &  & 5 & 0.63 & 2.02  & 0.48 & 1.02  \\
 &  & 10 & 0.80 & 3.60  & 0.53 & 1.16 \\
 \hline
\multirow[c]{9}{*}{large} & \multirow[c]{3}{*}{small} & 1 & 1.00 & 0.00  & 1.00 & 0.00  \\
 &  & 5 & 1.26 & 0.00  & 1.26 & 0.00  \\
 &  & 10 & 1.45 & 0.05  & 1.20 & 0.05  \\
 & \multirow[c]{3}{*}{medium} & 1 & 0.44 & 1.05 & 0.42 & 0.93  \\
 &  & 5 & 0.55 & 2.09 & 0.48 & 0.97  \\
 &  & 10 & 0.75 & 6.15 & 0.71 & 1.25  \\
 & \multirow[c]{3}{*}{large} & 1 & 0.42 & 1.73 & 0.41 & 1.64 \\
 &  & 5 & 0.56 & 1.57 & 0.47 & 1.39  \\
 &  & 10 & 0.89 & 1.78 & 0.69 & 1.64 \\
\end{tabular}
}
\caption{Solution quality information for the considered classes of NETLIB instances. From left to right we show the following information: the size of the dimension; the size of the number of constraints; the number of random columns for mutable parameter selection; the number of objective parameters which are changed compared to the present problem; the number of constraint parameters which are changed compared to the present problem.}\label{tbl:NETLIB_solutionInfo}
\end{table}

\begin{table}[h!]
\centering
\resizebox{0.9\columnwidth}{!}{
\begin{tabular}{lll|rrr|rrr|rr}
\multicolumn{3}{c|}{} & \multicolumn{3}{c|}{\eqref{eq:bilinear_relative_CE_formulation} with obj. $\sum_{j=1}^{n} \delta_j((c,A),(\hat c,\hat A))$} & \multicolumn{3}{c}{\eqref{eq:linearization_relative_CE} with obj. $\sum_{j=1}^{n} x_j\delta_j((c,A),(\hat c,\hat A))$} & \multicolumn{2}{c}{PP}\\
\hline
$n$ & $m$ & \makecell{\# mut.\\col.} & \makecell{Hit TL\\ (in \%)} & \makecell{$t$ \\ (in sec.)} & \makecell{$t$ infeas.\\ (in sec.)} & \makecell{Hit TL\\ (in \%)} & \makecell{$t$ \\ (in sec.)} & \makecell{$t$ infeas.\\ (in sec.)} & \makecell{$t$ \\ (in sec.)} & \makecell{$t$ setup \\ (in sec.)}\\
\hline
\multirow[c]{6}{*}{small} & \multirow[c]{3}{*}{small} & 1 & 0.00 & 0.19 & 0.13 & 0.00 & 0.14 & 0.13 & \multirow[c]{3}{*}{$0.13$} & \multirow[c]{3}{*}{$0.12$} \\
 &  & 5 & 1.00 & 29.19 & 0.13 & 0.00 & 0.15 & 0.13 & & \\
 &  & 10 & 2.00 & 43.97 & 0.14 & 0.00 & 0.17 & 0.15 & & \\
 & \multirow[c]{3}{*}{medium} & 1 & 0.00 & 0.42 & 0.31 & 0.00 & 0.31 & 0.30 & \multirow[c]{3}{*}{$0.32$} & \multirow[c]{3}{*}{$0.30$} \\
 &  & 5 & 1.00 & 16.09 & 0.31 & 0.00 & 0.33 & 0.31 & &\\
 &  & 10 & 3.00 & 53.40 & 0.32 & 0.00 & 0.36 & 0.33 & &\\
 \hline
\multirow[c]{9}{*}{medium} & \multirow[c]{3}{*}{small} & 1 & 0.00 & 0.79 & 0.78 & 0.00 & 0.53 & 0.77 & \multirow[c]{3}{*}{$0.53$} & \multirow[c]{3}{*}{$0.49$} \\
 &  & 5 & 0.00 & 1.49 & nan & 0.00 & 0.54 & nan & &\\
 &  & 10 & 4.00 & 109.16 & nan & 0.00 & 0.58 & nan & &\\
 & \multirow[c]{3}{*}{medium} & 1 & 0.00 & 0.99 & 0.45 & 0.00 & 0.50 & 0.44 & \multirow[c]{3}{*}{$0.50$} & \multirow[c]{3}{*}{$0.45$}\\
 &  & 5 & 0.00 & 11.28 & 0.44 & 0.00 & 0.49 & 0.44 & &\\
 &  & 10 & 1.00 & 32.76 & 0.45 & 0.00 & 0.51 & 0.45 & &\\
 & \multirow[c]{3}{*}{large} & 1 & 0.00 & 1.64 & 1.03 & 0.00 & 1.04 & 1.03 & \multirow[c]{3}{*}{$1.14$} & \multirow[c]{3}{*}{$1.06$}\\
 &  & 5 & 1.00 & 21.10 & 1.00 & 0.00 & 1.05 & 1.01 & &\\
 &  & 10 & 2.00 & 51.35 & 1.01 & 0.00 & 1.08 & 1.04 & &\\
 \hline
\multirow[c]{9}{*}{large} & \multirow[c]{3}{*}{small} & 1 & 0.00 & 11.91 & 8.40 & 0.00 & 11.10 & 8.27 & \multirow[c]{3}{*}{$8.22$} & \multirow[c]{3}{*}{$7.98$}\\
 &  & 5 & 0.00 & 11.43 & 8.48 & 0.00 & 10.50 & 8.28 & &\\
 &  & 10 & 0.00 & 12.17 & 7.27 & 0.00 & 11.08 & 7.19 & &\\
 & \multirow[c]{3}{*}{medium} & 1 & 0.00 & 2.17 & 1.40 & 0.00 & 0.90 & 1.36 & \multirow[c]{3}{*}{$1.34$} & \multirow[c]{3}{*}{$1.24$}\\
 &  & 5 & 0.00 & 3.48 & 1.39 & 0.00 & 0.92 & 1.35 & &\\
 &  & 10 & 0.00 & 15.65 & 1.42 & 0.00 & 0.95 & 1.38 & &\\
 & \multirow[c]{3}{*}{large} & 1 & 1.00 & 94.58 & 28.78 & 0.00 & 4.90 & 5.36 & \multirow[c]{3}{*}{$17.35$} & \multirow[c]{3}{*}{$3.18$} \\
 &  & 5 & 3.00 & 117.39 & 24.37 & 0.00 & 4.42 & 5.40 & &\\
 &  & 10 & 6.00 & 202.81 & 27.19 & 0.00 & 4.95 & 4.90 & &\\
\end{tabular}
}
\caption{Solution time information for the considered classes of NETLIB instances. We show from left to right the following values: the size of the dimension; the size of the number of constraints; the number of random columns for mutable parameter selection; the amount of instances for which the solution method hit the time limit of $1800$ seconds; the solution time $t$ in seconds (averaged over all instances which are not infeasible); the time $t$ which was needed to detect infeasibility (averaged over all instances which are infeasible); the time $t$ in seconds to solve the present problem (PP); the time $t$ to set up the problem formulation for the present problem.}\label{tbl:NETLIB_solutionTimeInfo}
\end{table}

\clearpage

\bibliographystyle{apalike}
\bibliography{references}

\section*{Appendix}

\subsection*{Proof of Theorem \ref{thm:closedness_relativeCE}}
 \ \\
 \begin{proof}
To prove the result we show that every convergent sequence of feasible points in the $(c,A,b)$-space has a limit which is also feasible. W.l.o.g. we may assume that the objective parameters $c$ are not mutable, since we can always shift them into the constraints by using an epigraph reformulation. Furthermore, w.l.o.g we may assume that the righ-hand-side parameters $b$ are not mutable since we can introduce a new variable $x_{n+1}$ and rewrite the constraint system as

Assume we have an infinite converging sequence $\{A^t\}_{t\in \N}$ of constraint matrices for which RCEP is feasible. We denote the limit of the sequence as $\bar A$. This limit lies in $\mathcal H$ since it is closed. Then for every $t$ there exists a point $x^t \in \mathcal D(\hat x)$ for which $A^tx^t\ge b$. Since $\mathcal D(\hat x)$ is compact, there exists a converging subsequence $\{x^t\}_{t\in I}$ with limit $\tilde x\in \mathcal D(\hat x)$. It holds
\[
\bar A \tilde x = \lim_{t\to \infty, t\in I} A^t \lim_{t\to \infty, t\in I} x^t = \lim_{t\to \infty, t\in I} A^t x^t \ge b
\]
where the latter inequality holds since $A^tx^t\ge b$ for all $t$ and the set $\{v: v\ge b\}$ is closed. By the same argumentation $\tilde x$ fulfills all other constraints of RCEP and hence $\bar A$ is feasible for RCEP which proves the closedness of the feasible region. Since $\mathcal H$ is bounded by assumption this proves compactness.
\end{proof}

\subsection*{Experimental Results Diet Problem}
Table \ref{tbl:CEs_Nutrients_Diet} shows the optimal relative and weak CEs for the reduced version of the diet problem studied in \ref{sec:reduced_diet_problem}.
\begin{table}[h!]
    \centering
    \resizebox{0.9\columnwidth}{!}{
    \begin{tabular}{c|c|c|c|c}
        \textbf{Supplier} & \multicolumn{2}{c|}{\textbf{relative CEs}} & \multicolumn{2}{c}{\textbf{weak CEs}}   \\
        \hline
        &CE&t (in s)&CE&t (in s) \\
        \hline
        \hline
        $1$ & \makecell{$p_{\text{wheat}}: 300.0 \to 126.2$ \\
        Fat(g) in Beans: $1.2 \to 2.4$ \\
Fat(g) in Wheat: $1.5 \to 3.0$ \\
} & $0.6$ &\makecell{$p_{\text{wheat}}: 300.0 \to 151.8$ \\
$p_{\text{beans}}: 800.0 \to 647.6$ \\
Energy(kcal) in Wheat : $330.0 \to 411.7$\\
Fat(g) in Beans : $1.2 \to 2.4$\\
Fat(g) in Wheat : $1.5 \to 3.0$\\
Iron(mg) in Beans : $8.2 \to 8.8$\\
Iron(mg) in Wheat : $4.0 \to 0.0$\\
ThiamineB1(mg) in Wheat : $0.3 \to 0.0$\\
RiboflavinB2(mg) in Beans : $0.22 \to 0.4$\\
RiboflavinB2(mg) in Wheat : $0.07 \to 0.0$\\
RiboflavinB2(mg) in WSB : $0.6 \to 0.7$\\
Folate(ug) in Beans : $180.0 \to 139.5$\\
Folate(ug) in Wheat : $51.0 \to 0.0$
} & $3600.0$ \\
        \hline
        $2$ & \makecell{$p_{\text{wheat}}: \infty \to 0.0$ \\ VitaminA(ug) in CSB: $500.0 \to 521.0$} & $0.4$ & \makecell{
        $p_{\text{beans}}: 1433.5 \to 828.3$ \\
        $p_{\text{csb}}: 799.0 \to 774.5$ \\
        $p_{\text{wheat}}: \infty \to 22.5$ \\
Fat(g) in Beans : $1.2 \to 2.4$\\
Fat(g) in CSB : $6.0 \to 12.0$\\
Fat(g) in Wheat : $1.5 \to 2.2$\\
Calcium(mg) in Wheat : $36.0 \to 47.4$\\
Iron(mg) in Beans : $8.2 \to 9.0$\\
Iron(mg) in CSB : $18.5 \to 13.0$\\
Iron(mg) in Wheat : $4.0 \to 0.0$\\
RiboflavinB2(mg) in Beans : $0.22 \to 0.4$\\
RiboflavinB2(mg) in CSB : $0.5 \to 1.0$\\
RiboflavinB2(mg) in Wheat : $0.07 \to 0.0$\\
NicacinB3(mg) in Beans : $2.1 \to 2.06$\\
NicacinB3(mg) in CSB : $6.8  \to 7.0$\\
NicacinB3(mg) in Wheat : $5.0 \to 0.3$\\
Folate(ug) in Beans : $180.0 \to 159.8$\\
Folate(ug) in Wheat : $51.0 \to 0.0$
} & $3600.0$\\
        \hline
        $3$ & \makecell{$p_{\text{wheat}}: 300.0 \to 126.2$ \\
        Fat(g) in Beans: $1.2 \to 2.4$\\
Fat(g) in Wheat: $1.5 \to 3.0$} & $0.8$ & \makecell{
$p_{\text{beans}}: 800.0 \to 694.9$ \\
$p_{\text{wheat}}: 300.0 \to 101.3$ \\
$p_{\text{wheat flour}}: 500.4 \to 499.2$ \\
Energy(kcal) in Wheat : $330.0 \to 332.7$\\
Fat(g) in Beans : $1.2 \to 2.4$\\
Fat(g) in Wheat : $1.5 \to 3.0$\\
Calcium(mg) in Wheat : $36.0 \to 34.2$\\
Iron(mg) in Wheat : $4.0 \to 0.2$\\
ThiamineB1(mg) in Wheat : $0.3 \to 0.1$\\
RiboflavinB2(mg) in Beans : $0.22 \to 0.4$\\
RiboflavinB2(mg) in Wheat : $0.07 \to 0.0$\\
Folate(ug) in Beans : $180.0 \to 140.2$\\
Folate(ug) in Wheat : $51.0 \to 0.0$
} & $3600.0$ \\
        \hline
        $4$ & \makecell{$p_{\text{wheat}}: \infty \to 26.3$ \\ Fat(g) in Wheat: $1.5\to 3.0$} & $0.5$ & \makecell{
        $p_{\text{wheat}}: \infty \to 87.6$ \\
$p_{\text{beans}}: 1151.1 \to 734.4$ \\
Energy(kcal) in Wheat : $330.0 \to 338.9$\\
Fat(g) in Beans : $1.2 \to 2.4$\\
Fat(g) in Oil : $100.0 \to 74.7$\\
Fat(g) in Wheat : $1.5 \to 3.0$\\
Iron(mg) in Beans : $8.2 \to 8.8$\\
Iron(mg) in Wheat : $4.0 \to 0.0$\\
RiboflavinB2(mg) in Beans : $0.22 \to 0.4$\\
RiboflavinB2(mg) in Wheat : $0.07 \to 0.0$\\
Folate(ug) in Beans : $180.0 \to 140.2$\\
value Folate(ug) in Wheat : $51.0 \to 0.0$
} & $3600.0$ \\
        \hline
        $5$ & $p_{\text{wheat}}: \infty \to 0.0$ & $0.4$ & \makecell{
        $p_{\text{wheat}}: \infty \to 85.1$\\
        $p_{\text{beans}}: 1383.2 \to 743.7$ \\
Energy(kcal) in Wheat : $330.0 \to 311.3$\\
Fat(g) in Beans : $1.2 \to 2.4$\\
Fat(g) in CSB : $6.0 \to 12.0$\\
Fat(g) in Milk : $1.0 \to 2.0$\\
Fat(g) in Wheat : $1.5 \to 3.0$\\
Calcium(mg) in Wheat : $36.0 \to 31.6$\\
Iron(mg) in Beans : $8.2 \to 8.8$\\
Iron(mg) in Wheat : $4.0 \to 0.0$\\
VitaminA(ug) in Milk : $280.0 \to 281.4$\\
RiboflavinB2(mg) in Beans : $0.22 \to 0.4$\\
RiboflavinB2(mg) in CSB : $0.5 \to 0.7$\\
RiboflavinB2(mg) in Milk : $1.2 \to 0.9$\\
RiboflavinB2(mg) in Wheat : $0.07 \to 0.0$\\
Folate(ug) in Beans : $180.0 \to 140.2$\\
Folate(ug) in Wheat : $51.0 \to 0.0$
} & $3600.0$ \\
        \hline
    \end{tabular}}
    \caption{Best known counterfactual explanations regarding changes in prices and nutrients after a time limit of one hour.}
    \label{tbl:CEs_Nutrients_Diet}
\end{table}

\end{document}